\documentclass[11pt]{article}
\usepackage{mathrsfs}
\usepackage{amsfonts}
\usepackage{amssymb,amsmath,amsthm}
\newcommand{\R}{\mathbb{R}}

\newcommand{\N}{\mathbb{N}}

\newtheorem{theorem}{Theorem}[section]

\newtheorem{lemma}[theorem]{Lemma}

\theoremstyle{definition}

\newtheorem*{notation}{Notation}

\begin{document}
	\title{\LARGE\bf{ Normalized solutions for nonautonomous Schr\"{o}dinger-Poisson equations}$\thanks{{\small This work was partially supported by NSFC(11901532).}}$ }
	\date{}
	\author{  Yating Xu, Huxiao Luo$\thanks{{\small Corresponding author. E-mail: 1151980124@qq.com (Y. Xu), luohuxiao@zjnu.edu.cn (H. Luo).}}$\\
		\small Department of Mathematics, Zhejiang Normal University, Jinhua, Zhejiang, 321004, P. R. China}
	\maketitle
\begin{center}
\begin{minipage}{13cm}
\par
\small  {\bf Abstract:} In this paper, we study the existence of normalized solutions for the nonautonomous Schr\"{o}dinger-Poisson equations
\begin{equation}\nonumber
-\Delta u+\lambda u +\left(\vert x \vert ^{-1} * \vert u \vert ^{2} \right) u=A(x)|u|^{p-2}u,\quad \text{in}~\R^3,
\end{equation}
where $\lambda\in\R$, $A \in L^\infty(\R^3)$ satisfies some mild conditions. Due to the nonconstant potential $A$, we use Pohozaev manifold to recover the compactness for a minimizing sequence. For $p\in (2,3)$, $p\in(3,\frac{10}{3})$ and $p\in(\frac{10}{3}, 6)$, we adopt different analytical techniques to overcome the difficulties due to the presence of three terms in the corresponding energy functional which scale differently, respectively.

\end{minipage}
\end{center}

{\section{Introduction}}
\setcounter{equation}{0}
We consider the nonautonomous Schr\"{o}dinger-Poisson system of the type
\begin{equation}\label{t1.1.0}
\left\{
\begin{array}{ll}
\aligned
&i\partial_t\psi=-\Delta \psi+\left(\frac{1}{|x|}\ast |\psi|^{2}\right)\psi-A(x)|\psi|^{p-2}\psi,\quad t\in\R,~ x\in\R^3,\\
&\psi(0,x)=\psi_0(x)\in H^1(\R^N,\mathbb{C}),
\endaligned
\end{array}
\right.
\end{equation}
where $\psi(t, x) : \R^+\times\R^3\to\mathbb{C}$ is the wave function, $p\in(2,6)$, $A \in C^1(\R^3,\R)$, $\frac{1}{|x|}\ast |\psi|^{2}$ represents an internal potential for a nonlocal self-interaction of the
wave function $\psi$. The Schr\"{o}dinger-Poisson system is obtained by
approximation of the Hartree-Fock equation describing a quantum mechanical system of many particles,
see for instance \cite{4,17,15}. The system has been extensively studied
and is quite well understood in the autonomous case $A\equiv1$. In particular, variational methods
are employed to derive existence and multiplicity results of standing waves solutions
\cite{2,1,20} and for standing wave solutions with prescribed $L^2-$norm
\cite{jlt2020nonautonomous,jg2011scaling,nonexistence,sm2014stablestandingwaves}.

In this article, we focus on the nonautonomous case: $A(x)$ is not necessarily a constant function. As far as we know, there is no result about the existence of normalized solutions for the nonautonomous Schr\"{o}dinger-Poisson systems. We study the existence of the standing wave solutions of \eqref{t1.1.0}, which have the form $\psi(t, x) = e^{i\lambda t}u(x)$, where $\lambda\in\R$
is the frequency and $u(x)$ solves the following elliptic equation
\begin{equation}\label{1.1.0}
\left\{
\begin{array}{ll}
\aligned
&-\Delta u+\lambda u +\left(\vert x \vert ^{-1} * \vert u \vert ^{2} \right) u=A(x)|u|^{p-2}u,\quad \text{in}~\R^3, \\
&\int_{\mathbb{R}^{3}}|u|^{2}dx=\int_{\R^N}|\psi_0|^2dx:=c,\quad u\in H^1(\mathbb{R}^3),
\endaligned
\end{array}
\right.
\end{equation}
where the constraint $\int_{\mathbb{R}^{N}}|u|^{2}dx=c$ is natural due to the conservation of mass. By variational method, prescribed $L^2-$norm solution
(normalized solution) of $(\ref{1.1.0})$ can be searched as critical points of $I(u)$ constrained on $S(c)$, and the parameter $\lambda$ appears as a Lagrangian multiplier, where
\begin{equation}\label{I}
	I(u)=\frac{1}{2}\int_{\R^3}|\nabla u |^2dx+\frac{1}{4}\int_{\R^3}\left(\vert x \vert ^{-1} * \vert u \vert ^{2} \right)|u|^{2}dx-\frac{1}{p}\int_{\R^3}A(x)|u|^{p}dx,
\end{equation}
and
$$S(c)=\left\{u \in H^{1}\left(\R^3\right) :\int_{\R^3}|u|^2dx=c\right\}.
$$
To study (\ref{1.1.0}) variationally, we need to recall the following Hardy--Littlewood--Sobolev inequality.
\begin{lemma}(\cite[Theorem 4.3]{MR1817225}) Let $s,r>1$ and $0<\mu<N$ with $1/s+\mu/N+1/r=2$, $f\in L^{s}(\R^N)$ and $h\in L^{r}(\R^N)$. There exists a sharp constant $C(N,\mu,s,r)$, independent of $f,h$, such that
\begin{equation}\label{HLS1}
\int_{\mathbb{R}^{N}}\int_{\mathbb{R}^{N}}\frac{f(x)h(y)}{|x-y|^{\mu}}dxdy\leq C(N,\mu,s,r) \|f\|_{s}\|h\|_{r}.
\end{equation}
Here, $\|\cdot\|_s$ denotes the usual Lebesgue norm in $L^s(\R^N)$.
\end{lemma}
By (\ref{HLS1}) and Sobolev inequality, the integral
$$
\int_{\mathbb{R}^{3}}\int_{\mathbb{R}^{3}}\frac{ |u(x)|^2|u(y)|^2}{|x-y|} \mathrm{d} x\mathrm{d} y \leq \|u\|^4_{12/5}
$$
is well defined in $H^1(\R^N)$.

For $L^2-$subrcritical case $p\in (2,\frac{10}{3})$ and for $L^2-$superrcritical case $p\in (\frac{10}{3},6)$, the energy functional $I$ possess different structures. Indeed, by Gagliardo-Nirenberg inequality it is easy to see that $$\gamma(c):=\inf\limits_{u\in S(c)} I(u)>-\infty$$\label{Inf_c} for $p\in (2,\frac{10}{3})$. Specifically, for the Coulomb-Sobolev critical case $p=3$, the nonlocal term $\int_{\R^3}(|x|^{-1}\ast u^2)u^2dx$ and the local autonomous term $\int_{\R^3}|u|^pdx$ have the same scaling for the preserving $L^2-$norm transformation $u^t(\cdot)= t^\frac{3}{2} u(t\cdot)$. For $p\in (2,3)$ and $p\in(3,\frac{10}{3})$, by different analytical techniques, we find a global minimizer for $\gamma(c)$, which is also a normalized (ground state) solution to \eqref{1.1.0}.
\begin{theorem}\label{1 th}
Suppose that
\begin{itemize}
			\item[(A1)] $A \in C^1(\R^3)$, $\lim\limits_{|x|\to\infty} A(x) = A_{\infty}\in (0,\infty)$ and $A(x) \geq A_{\infty} $ for all $x\in \R^3 $;
\item[(A2)] $ t \mapsto \frac{3}{2} (p-2)A(tx) - \nabla A(tx)\cdot (tx)$ is nonincreasing on  $t\in (0, \infty)$ for every $x\in \R^3$.
\end{itemize}
If $p\in(2,3)$, then for any $c\in(0,+\infty)$ equation $\eqref{1.1.0}$ possesses a normalized solution $(u_c, \lambda_c)$;
If $p\in(3,\frac{10}{3})$, then there exists a $\bar{c}>0$ such that for any $c\in[\bar{c},+\infty)$ equation $\eqref{1.1.0}$ possesses a normalized solution $(u_c, \lambda_c)$.
\end{theorem}
\begin{notation}
It is a challenging problem to prove the existence of normalized solutions for the Coulomb-Sobolev critical case $p=3$.
\end{notation}

For $L^2-$supercritical case $p\in (\frac{10}{3},6)$, $I(u)$ has no lower bound on $S(c)$.
Inspired by \cite{sx2020nonautonomous},
we establish the existence of normalized solutions to the non-autonomous Schr\"{o}dinger-Poisson equation $\eqref{1.1.0}$ by using the Pohozaev manifold
\begin{equation}\label{Gamma def}
	\mathcal{P}(c)=\left\{u\in S(c): P(u):=\dfrac{\partial}{\partial t}I(u^t)\big|_{t=1}=0\right\},
\end{equation}
where $u^t(x):=t^{\frac{3}{2}}u(tx)$ for all $t>0$ and $x\in \R^3$, and $u^t\in S(c)$ if $u\in S(c)$.
\begin{theorem}\label{2 th}
Let $\frac{10}{3}<p<6$. Assume that $(A1)$, $(A2)$ and
\begin{itemize}
	\item[(A3)]$ t \mapsto t^{\frac{3}{2}}A(tx)$is nondecreasing on $(0, \infty)$ for every $x\in \R^3$.
\end{itemize}
Then $\forall c>0$, $\eqref{1.1.0}$ possesses a solution $(\bar{u}_c,\lambda_c)\in S(c)\times \mathbb{R}$ such that
	\begin{equation}\nonumber
		I(\bar{u}_c)=\inf\limits_{u\in \mathcal{P}(c)}I(u)=\inf\limits_{u\in S(c)}\max\limits_{t>0}I(u^t)>0.
	\end{equation}
\end{theorem}
Besides $A\equiv$constant, there are indeed many functions which satisfy
$(A_1), (A_2)$ and $(A_3)$. For example
\begin{itemize}
\item[($i$)] $A_1(x) = 1 + ce^{-\tau|x|}$ with some $c>0$ and $\tau> 0$;
\item[($ii$)] $A_2(x) = 1 + \frac{c}{1+|x|}$ with some $c>0$.
\end{itemize}

This article is organized as follows. In sections 2-3, we prove Theorem \ref{1 th} and Theorem \ref{2 th} respectively.

We make use of the following notation:
\begin{itemize}
\item [(i)] $2^*$ denotes the Sobolev critical exponent, and $2^*=6$ for $N=3$;
	\item [(ii)] $C,C_i,i=1,2,\cdots$ will be used repeatedly to represent various normal numbers, the exact value of which is irrelevant;
	\item [(iii)]$o(1)$ denotes the infinitesimal as $n\to +\infty$.
\end{itemize}

\vskip4mm
{\section{$L^2$ -subcritical case $p\in(2,\frac{10}{3})$}}
\setcounter{equation}{0}
In the section, we prove Theorem \ref{1 th} under $p\in(2,\frac{10}{3})$ and conditions (A1)-(A2).
Since we consider equation $\eqref{1.1.0}$ in the whole space $\R^3$, the Sobolev embedding $H^1(\R^3)\hookrightarrow L^p(\R^3)$ isn't compact.
To solve the lack of compactness, we consider the limiting equation of $\eqref{1.1.0}$:
\begin{equation}\nonumber
	-\Delta u+\lambda u +\left(\vert x \vert ^{-1} * \vert u \vert ^{2} \right) u=A_{\infty}|u|^{p-2}u,\; u\in H^1(\R^3).
\end{equation}
The energy functional is defined as follows:
\begin{equation}\nonumber
	I_{\infty}(u)=\dfrac{1}{2}\|\nabla u\|^2_2+\frac{1}{4}\int_{\R^3}\int_{\R^3}\frac{|u(x)|^2|u(y)|^2}{|x-y|}dxdy-\frac{A_{\infty}}{p}\int_{\R^3}|u(x)|^pdx.
\end{equation}
Similar to $\eqref{Gamma def}$, we define the Pohozaev manifold $\mathcal{P}_{\infty}(c)$ as
\begin{equation}\nonumber
	\mathcal{P}_{\infty}(c)=\left\{u\in S(c): P_{\infty}(u):=\dfrac{\partial}{\partial t}I_{\infty}(u^t)\big|_{t=1}=0\right\}.
\end{equation}
Since $A(x)=A_{\infty}$ satisfies (A1), (A2), all the following conclusions on $I(u)$ are also true for $I_{\infty}$.

Now, we give the classical Brezis-Lieb lemma $\cite{brezislieb}$ and a nonlocal version of Brezis-Lieb lemma, which will be used in
both $L^2$-subcritical case and $L^2$-supercritical case.
\begin{lemma}\label{1 le 2.1} (\cite{brezislieb})
	Let $\{u_n\}\subset L^p(\R^N)$, $p\in[1,+\infty)$. If
\begin{itemize}
	\item [(i)] $\{u_n\}$ is bounded in $L^p(\R^N)$;
	\item [(ii)]$u_n(x)\to u(x)$ a. e. on $\R^N$, as $n\to +\infty$.
\end{itemize}
Then
\begin{equation}\nonumber
	\int_{\R^N}|u_n|^p dx-\int_{\R^N}|u_n-u|^p dx =\int_{\R^N}|u|^p dx+o(1).
\end{equation}
\end{lemma}

\begin{lemma}\label{1 le 2.2}
Let $q \in [2,2^*] $and $A\in L^{\infty}(\R^N)$. If $u_n \rightharpoonup u$ in $H^1(\R^N)$, then
\begin{equation}\label{2.2}
	\int_{\R^N}A(x)|u_n|^p dx-\int_{\R^N}A(x)|u_n-u|^p dx =\int_{\R^N}A(x)|u|^p dx+o(1),
\end{equation}
\begin{equation}\label{2.3}
	\int_{\R^N}| \nabla  u_n|^2 dx-	\int_{\R^N}  |\nabla  u_n-u|^2 dx=\int_{\R^N}|\nabla  u|^2 dx+o(1),
\end{equation}
\begin{equation}\label{2.4}
	\int_{\R^N}(|x|^{-1}*|u_n|^2)|u_n|^2 dx-\int_{\R^N}(|x|^{-1}*|u_n-u|^2)|u_n-u|^2 dx =\int_{\R^N}(|x|^{-1}*|u|^2)|u|^2 dx+o(1).
\end{equation}
\end{lemma}
\begin{proof}
From the proof of Lemma \ref{1 le 2.1} (see \cite[Lemma 1.32]{brezislieb}), we know that
$$|u_n|^p-|u|^p-|u_n-u|^p\to0~~\text{in}~ L^1(\R^N),$$ as $n\to \infty$.
Along with $A(x)\in L^{\infty}(\R^N)$ we have
$$\left|\int_{\R^N}A(x)(|u_n|^p-|u_n-u|^p-|u|^p) dx\right|\le \|A\|_{\infty}\int_{\R^3}\left||u_n|^p-|u_n-u|^p-|u|^p\right|dx=o(1),$$
then $\eqref{2.2}$ holds.

$\eqref{2.3}$ is  a corollary of Lemma \ref{1 le 2.1} and Sobolev embedding theorem.

$\eqref{2.4}$ is the Brezis-Lieb lemma with Riesz potential, see \cite[(3.6)]{MS} (also \cite[Lemma 2.2]{luowang} and \cite[Lemma 2.2]{LRT}).
\end{proof}

In order to estimate the lower bound of $I(u)$, we need the following inequality \cite{nonexistence}.
\begin{lemma}\cite[Lemma 2.2]{nonexistence}\label{1 lec(u)}
	There exists a constant $C>0$ such that for all $u\in S(c)$
	\begin{equation}\label{1 equ2.1 c(u)}
		\int_{\R^3}\int_{\R^3}\frac{|u(x)|^2|u(y)|^2}{|x-y|}dxdy \ge -\frac{1}{16\pi}\|\nabla  u\|^2_2 + C\|u\|^3_3.
	\end{equation}
\end{lemma}

\begin{lemma}\label{1 lem3.1} Under the assumption (A1)-(A2), we have
\begin{equation}\label{A2 equ1.1}
		\begin{aligned}
			\phi (t,x)=-t^{-\frac{3}{2}(p-2)}[A(x)-A(tx)]+
			\frac{2\left(t^{-\frac{3}{2}(p-2)}-1\right)}{3(p-2)}\nabla A(x)\cdot x \geq 0, \; \forall t>0,x\in \R^3,
		\end{aligned}
	\end{equation}
	\begin{equation}\label{A2 equ1.2}
		t\mapsto A(tx)\ is \ nonincreasing \ on~(0,\infty) \; \forall x\in\R^3,
	\end{equation}
	\begin{equation}\label{A2 equ1.3}
		\begin{aligned}
			-\nabla A(x)\cdot x\geq 0,\; \forall x\in \R^3, \quad
			\text{and}~ -\nabla A(x)\cdot x \to 0,\; as |x|\to \infty.
		\end{aligned}
	\end{equation}
\end{lemma}
\begin{proof}
	First, we take the derivative of $\phi(t,x)$, and with respect to (A2), for any $x\in \R^3$ we have
	\begin{equation}\nonumber
		\begin{aligned}
			\dfrac{\partial }{\partial t}\phi(t,x)
			&=t^{-\frac{3}{2}(p-2)-1}\left[\dfrac{3}{2}(p-2)A(x)-\dfrac{3}{2}(p-2)A(tx)+\nabla A(tx)\cdot tx-\nabla A(x)\cdot x\right]\\
		&\left\{ \begin{aligned}
			&\geq 0,\; t\geq 1;\\
			&\le 0,\; 0\le t\le 1.
		\end{aligned}  \right.
	\end{aligned}
\end{equation}
	This implies that $\phi(t,x)\geq \phi(1,x)=0$ for all $t>0$ and $x\in \R^3$, i.e. \eqref{A2 equ1.1} holds.\par
	Next, let $t\to +\infty$ in \eqref{A2 equ1.1}, we obtain $-\nabla A(x)\cdot x \geq 0$ for all $x\in \R^3$, which implies \eqref{A2 equ1.2}.

Finally, we put $t=\frac{1}{2}$ into \eqref{A2 equ1.1}, by (A1) then we have
	\begin{equation}\nonumber
		\begin{aligned}
			0\geq\frac{4\left(2^{\frac{3}{4}(p-2)}-1\right)}{3(p-2)}\nabla A(x)\cdot x \geq  2^{\frac{3}{4}(p-2)}\left[A(x)-A\left(\frac{x}{2}\right)\right].
		\end{aligned}
	\end{equation}
Let $|x| \to \infty$, we get \eqref{A2 equ1.3}.
\end{proof}

For convenience, we set
$$B(u):=\int_{\R^3}|\nabla u |^2dx,\quad C(u):=\int_{\R^3}\left(\vert x \vert ^{-1} * \vert u \vert ^{2} \right)|u|^{2}dx.$$
Then
\begin{equation*}
		I(u)=\dfrac{1}{2}B(u)+\dfrac{1}{4}C(u)-\frac{1}{p}\int_{\R^3}A(x)|u|^pdx.
	\end{equation*}
\begin{lemma}\label{1 well defined}
$\gamma (c)=\inf_{u \in S(c)}I(u)$ is well defined of all $c>0$ and $\gamma (c)\le 0$.
\end{lemma}
\begin{proof}
By (A1) and Gagliardo-Nirenberg inequality, we have
$$\int_{\R^3}A(x)|u|^pdx\leq C \|\nabla  u\|^{\frac{3}{2}(p-2)}_2 \|u\|_2^{3-\frac{p}{2}}.$$
Then by  \eqref{1 equ2.1 c(u)}, we derive at
	\begin{equation}\label{1 equ3.1 bound}
		I(u) \geq \left(\frac{1}{2}-\frac{1}{64\pi}\right)\|\nabla  u\|^2_2+ C\|u\|^3_3- C'\|\nabla  u\|^{\frac{3}{2}(p-2)}_2.
	\end{equation}
	Thus $\gamma (c)=\inf_{u \in S(c)}I(u)$ is well defined of all $c>0$ and $\gamma(c)>-\infty$.\par
	For every $u\in S(c)$, let $u^t= t^\frac{3}{2} u(tx)$, then $\|u^t\|_2=\|u\|_2=c$,
\begin{equation}\nonumber
		\begin{aligned}
&B(u^t)=\int_{\R^3}|\nabla u^t|^2dx=t^{2} B(u),\\
&C(u^t)=\int_{\R^3}\int_{\R^3}\frac{|u^t(x)|^2|u^t(y)|^2}{|x-y|}dxdy=tC(u).
		\end{aligned}
	\end{equation}
Thus, by (A1) we see that
\begin{equation}\nonumber
		\begin{aligned}
			I(u^t)&=\frac{1}{2}B(u^t) +\frac{1}{4}C(u^t)-\frac{1}{p}t^{\frac{3}{2}(p-2)}\int_{\R^3}A(t^{-1}x)|u|^pdx\\
			&\le \frac{1}{2}B(u^t) +\frac{1}{4}C(u^t)-\frac{1}{p}A_{\infty}t^{\frac{3}{2}(p-2)}\int_{\R^3}|u|^pdx \\
			&< \frac{1}{2}B(u^t) +\frac{1}{4}C(u^t)  \to 0,\quad\text{as}~t\to0.
		\end{aligned}
	\end{equation}
This implies $\gamma (c)\leq 0$.
\end{proof}

Now, we prove the subadditivity of $\gamma(c)$. Due to the nonlocality of the term $\int_{\R^3}(I_2\ast u^2)u^2dx$, the subadditivity of $\gamma(c)$ cannot be obtained directly.


\begin{lemma}\label{1 lesubadditivity}
	For each $c',c>0$, we have
\begin{equation}\label{1 equ4.1 subadditivity}
		\gamma (c+c')\le \gamma(c)+\gamma(c').
	\end{equation}
\end{lemma}
\begin{proof}
We fix $\varepsilon> 0$. By the definition of $\gamma(c)$ and $\gamma(c')$, there exist $u\in S(c)\cap C^\infty_0(\R^3)$
and $v\in S(c')\cap C^\infty_0(\R^3)$ such that
$$I(u)\leq\gamma(c) + \varepsilon,\quad
I(v)\leq \gamma(c')+ \varepsilon.$$
Since $u$ and $v$ have compact support, by using parallel translation, we can assume
$$\text{supp} u \cap \text{supp} v = \emptyset,\quad dist\{\text{supp} u,\text{supp} v \}\geq\frac{1}{\varepsilon}.$$
Thus $u + v \in S(c+c'),$
\begin{equation*}
\int_{\R^3}A(x)|u(x)+v(x)|^pdx=\int_{\R^3}A(x)|u(x)|^pdx+\int_{\R^3}A(x)|v(x)|^pdx,
\end{equation*}
and
\begin{equation*}
\aligned
&\int_{\R^3}\int_{\R^3}\frac{|u(x)+v(x)|^2|u(y)+v(y)|^2}{|x-y|}dxdy\\
=&\int_{\R^3}\int_{\R^3}\frac{|u(x)|^2|u(y)|^2}{|x-y|}dxdy+\int_{\R^3}\int_{\R^3}\frac{|v(x)|^2|v(y)|^2}{|x-y|}dxdy
+2\int_{supp{v}}\int_{supp{u}}\frac{|u(x)|^2|v(y)|^2}{|x-y|}dxdy\\
\leq &\int_{\R^3}\int_{\R^3}\frac{|u(x)|^2|u(y)|^2}{|x-y|}dxdy+\int_{\R^3}\int_{\R^3}\frac{|v(x)|^2|v(y)|^2}{|x-y|}dxdy+2\varepsilon\int_{\R^3}|u(x)|^2dx\int_{\R^3}|v(y)|^2dy.
\endaligned
\end{equation*}
Therefore
$$ \gamma (c+c')\leq I(u + v) \leq I(u) + I(v)+2cc'\varepsilon\leq \gamma(c)+\gamma(c')+ (2+2cc')\varepsilon.$$
As $\varepsilon\to0$, we get \eqref{1 equ4.1 subadditivity}.
\end{proof}

\begin{lemma} \label{1 gammainfty} Let $\gamma_{\infty}(c)=\inf_{u \in S(c)}I_{\infty}(u)$, then
	$\gamma(c) \le \gamma_{\infty}(c)$ for all $c>0$.
\end{lemma}
\begin{proof}
For every $c>0$, we choose $\{u_n\} \subset S(c)$ such that $I_{\infty}(u_n)< \gamma_{\infty} +\dfrac{1}{n}$. By (A1), we have
	\begin{equation}\nonumber
		\gamma(c) \le I(u_n) \le I_{\infty}(u_n) <\gamma_{\infty}(c)+\frac{1}{n},
	\end{equation}
	which implies that $\gamma(c)\le \gamma_{\infty}(c)$.
\end{proof}

\begin{lemma}\label{1 continuous}
$\gamma(c)$ is  continuous on $(0,\infty)$.
\end{lemma}
\begin{proof}
	For any $c>0$, let $c_n\to c$ and $u_n \subset S(c_n)$ such that $I(u_n)< \gamma(c_n) +\dfrac{1}{n}\le \dfrac{1}{n}$. Then \eqref{1 equ3.1 bound} and $2<p<\frac{10}{3}$ imply that $\{u_n\}$ is bounded in $H^1(\R^3)$. Hence,
	\begin{equation}\label{1 equ5.1 left}
		\begin{aligned}
			\gamma(c)&\le I\left(\sqrt{\dfrac{c}{c_n}}u_n\right)\\
			&=\dfrac{c}{2c_n}\|u_n\|_2^2+\dfrac{1}{4}\left(\dfrac{c}{c_n}\right)^2 \int_{\R^3}\int_{\R^3}\frac{|u_n(x)|^2|u_n(y)|^2}{|x-y|}dxdy-\dfrac{1}{p}\left(\dfrac{c}{c_n}\right)^{\frac{p}{2}}\int_{\R^3}A(x)|u_n|^pdx\\
			&=I(u_n)+o(1) \le \gamma(c_n) +o(1).
		\end{aligned}
	\end{equation}
	On the other hand, let  $\{v_n\}$ be any minimizing sequence with respect to $\gamma(c)$, we have
	\begin{equation}\label{1 equ5.2 right}
		\gamma(c_n)\le I(\sqrt{\dfrac{c_n}{c}}v_n) \le I(v_n)+o(1) = \gamma(c) +o(1).
	\end{equation}
From \eqref{1 equ5.1 left} and \eqref{1 equ5.2 right}, we derive at $\lim\limits_{n\to \infty}\gamma(c_n)=\gamma(c)$.
\end{proof}

In the following lemma, we use the strict subadditivity
\begin{equation}\label{1 20231008-e3}
	\gamma (c+c') < \gamma(c) +\gamma(c'),~~\forall c,c'\in \mathcal{U}
\end{equation}
and $\gamma(c)<0$ to show that $\gamma(c)$ is achieved for any $c\in\mathcal{U}$,
where $\mathcal{U}$ is some interval in $\R^+$.
The proof of \eqref{1 20231008-e3} and $\gamma(c)<0$ will be postponed to subsections 2.1 and 2.2, respectively in case $p\in(2,3)$ and case $p\in\left(3,\frac{10}{3}\right)$.
Indeed, $\mathcal{U}=(0,+\infty)$ for case $ p\in(2,3)$; $\mathcal{U}=[\bar{c},+\infty)$ for case $ p\in\left(3,\frac{10}{3}\right)$, where $\bar{c}>0$ is defined by \eqref{1.1 equ bar(c)} in subsection 2.2.

\begin{lemma}\label{1 achieved}
	Assume that \eqref{1 20231008-e3} and $\gamma(c)<0$ hold. Then $\gamma(c)$ is achieved for all $c\in \mathcal{U}$.
\end{lemma}
\begin{proof}
Let $\{u_n\}\subset S(c)$ be any minimizing sequence with respect to $\gamma(c)$. Then \eqref{1 equ3.1 bound} and $2<p<\frac{10}{3}$ imply that $\{u_n\} $ is bounded in $H^1 (\R^3)$. Up to a subsequence, there exists $\bar{u} \in H^1 (\R^3)$ such that
\begin{equation}\nonumber
	u_n\rightharpoonup \bar{u} \; in \; H^1(\R^3) , \ u_n\to \bar{u} \; in \; L_{loc}^q(\R^3)\;\text{for}~ q\in (2,2^*), \ u_n\to \bar{u} \ a.e.\ in\ \R^3.
\end{equation}

Case (I): $\bar{u}\neq 0$. From Lemma \ref{1 le 2.1}, Lemma \ref{1 le 2.2} and Lemma \ref{1 continuous}, we have
	 \begin{equation}\nonumber
		\begin{aligned}
			\gamma(c)&=\lim\limits_{n\to \infty}I(u_n) =I(\bar{u}) +\lim\limits_{n\to \infty}I(u_n-\bar{u}) \\
			&\ge \gamma(\|\bar{u}\|^2_2) + \lim\limits_{n\to \infty}\gamma(\|u_n - \bar{u}\|^2_2) = \gamma(\|\bar{u}\|^2_2) + \gamma(c-\|\bar{u}\|^2_2).
		\end{aligned}
	\end{equation}
	If $\|\bar{u}\|^2_2 < c$, then by \eqref{1 20231008-e3} we get
$$\gamma(c) \ge \gamma(\|\bar{u}\|^2_2) +\gamma(c-\|\bar{u}\|^2_2) > \gamma(c),$$
 a contradiction. Thus $\|\bar{u}\|^2_2 = c$ and $\bar{u}\in S(c)$. Then, combining $\|u_n\|_2\to\|\bar{u}\|_2$ and $u_n\rightharpoonup \bar{u}$ in $L^2(\R^3)$ we have $u_n\to \bar{u}$ in $L^2(\R^3)$. Moreover, by Gagliardo-Nirenberg inequality and the boundedness of $A(\cdot)$, we get
 $$\int_{\R^3}A(x)|u_n|^pdx\to\int_{\R^3}A(x)|\bar{u}|^pdx.$$
Hence, by the weak semicontinuity of norms we have
	\begin{equation}\nonumber
		\gamma(c)=\lim\limits_{n\to \infty}I(u_n)\geq I(\bar{u})\geq \gamma(c).
	\end{equation}
Therefore $\gamma(c)=I(\bar{u})$. This shows that $\gamma(c) $ is achieved at $\bar{u}\in S(c)$.\par
Case (II): $\bar{u}= 0$. By (A1), for any $\varepsilon>0$, there exists $R>0$ large enough such that
$$|A_{\infty}-A(x)|<\varepsilon~~ \forall x\in B^c_R(0).$$
Then by $u_n\to 0$ in $L_{loc}^q(\R^3)$ and Sobolev imbedding inequality we obtain
\begin{equation}\label{2 equa.2 third0limit}
		\begin{aligned}
			&\left|\int_{\R^3}(A_{\infty}-A(x))|u_n|^pdx\right|\\
			&\le \left|\int_{B_R(0)}\left(A_{\infty}-A(x)\right)|u_n|^pdx\right|+\left|\int_{B_R^c(0)}\left(A_{\infty}-A(x)\right)|u_n|^pdx\right|\\
	&\le o(1)+\varepsilon \int_{B_R^c(0)}|u_n|^pdx\le o(1)+ C\varepsilon.
		\end{aligned}
	\end{equation}
Therefore, by \eqref{2 equa.2 third0limit} and the arbitrariness of $\varepsilon$ we derive at
 \begin{equation}\label{1 equa.3 Iinftlimit}
		\lim\limits_{n\to\infty}I_{\infty}(u_n)=\lim\limits_{n\to\infty}I(u_n)= \gamma(c).
	\end{equation}
	Next, we show
\begin{equation}\label{1 equa.4 lion local lower limit}
		\delta :=\varlimsup\limits_{n \to \infty}\sup_{u\in \R^3}\int_{B_1(y)}|u_n|^2dx>0,
	\end{equation}
	where $B_1(y)=\{x\in \R^3: |x-y|<1\}$.
	Suppose otherwise, then by Lions' concentration compactness principle $\cite{lionI1984,lionII1984}$
$$\ u_n\to 0 \; in \; L^q(\R^3)\; for \; q\in (2,2^*).$$
This together with (A1) implies that
\begin{equation}\label{1 equa.5}
	\lim\limits_{n\to \infty}\int_{\R^3}A(x)|u_n(x)|^pdx=0.
	\end{equation}
	Then  by $\eqref{1 equa.5}$ and $\gamma(c)<0$, we get
\begin{equation}\nonumber
		0>\gamma(c)=\lim\limits_{n\to \infty}I(u_n)=\lim\limits_{n\to \infty}\left(\dfrac{1}{2}\|\nabla u_n\|^2_2+\frac{1}{4}\int_{\R^3}\int_{\R^3}\frac{|u_n(x)|^2|u_n(y)|^2}{|x-y|}dxdy \right)\geq 0,
	\end{equation}
a contradiction. Hence, $\delta >0$, and then there exists a sequence $\{y_n\}^{\infty}_{n=1} \subset \R^3$ such that \begin{equation}\label{1 equa.6}
		\int_{B_1(y_n)}|u_n|^2dx>\dfrac{\delta}{2}>0.
	\end{equation}
	Let $\hat{u}_n(x)=u_n(x+y_n)$, it is easy to see that $\|\hat{u}_n\|^2_2=\|u_n\|^2_2=c$, $\|\nabla\hat{u}_n\|^2_2=\|\nabla u_n\|^2_2$ and
\begin{equation}\label{1 equa.7}
		I_{\infty}(\hat{u}_n)\to \gamma(c).
	\end{equation}
Since $\hat{u}_n$ is bounded in $H^1 (\R^3)$, there exists $\hat{u} \in H^1 (\R^3)$ such that
	\begin{equation}\label{1 equa.8 functionrowconvergence2}
		\hat{u}_n\rightharpoonup \hat{u} \; in \; H^1(\R^3) , \ \hat{u}_n\to \hat{u} \; in \; L_{loc}^q(\R^3)\; q\in [2,2^*), \ \hat{u}_n\to \hat{u} \ a.e.\ in\ \R^3.
	\end{equation}
And by \eqref{1 equa.6}, we see that $\hat{u}\not\equiv0$.
Combining \eqref{1 equa.7}, \eqref{1 equa.8 functionrowconvergence2}, Lemma \ref{1 le 2.1}, Lemma \ref{1 le 2.2}, Lemma \ref{1 gammainfty} and  Lemma \ref{1 continuous}, we get
	\begin{equation}\label{1 equa.9}
		\begin{aligned}
			\gamma_{\infty}(c) &\geq \gamma(c) = \lim\limits_{n\to \infty} I_{\infty}(\hat{u}_n)=I_{\infty}(\hat{u})+\lim\limits_{n\to \infty}I_{\infty}(\hat{u}_n-\hat{u})\\
			&\geq \gamma_{\infty}(\|\hat{u}\|^2_2)+\lim\limits_{n\to \infty}\gamma_{\infty}(\|\hat{u}_n-\hat{u}\|^2_2) = \gamma_{\infty}(\|\hat{u}\|^2_2)+\gamma_{\infty}(c-\|\hat{u}\|^2_2).
		\end{aligned}
	\end{equation}
	If $\|\hat{u}\|^2_2 <c$, then $\eqref{1 equa.9}$ and \eqref{1 20231008-e3} imply
	\begin{equation}\nonumber
		\gamma_{\infty}(c) \geq \gamma_{\infty}(\|\hat{u}\|^2_2) + \gamma_{\infty}(c-\|\hat{u}\|^2_2) > \gamma_{\infty}(c),
	\end{equation}
	which is impossible. Thus $\|\hat{u}\|^2_2 = c$. Therefore
$$\hat{u}_n \to \hat{u}\; in \;L^q(\R^3)\; for \; 2\leq q < 2^*.$$
Thanks to the lower semicontinuity of norms, we obtain
\begin{equation}\nonumber
	 \gamma(c)=\lim\limits_{n\to \infty}I_{\infty}(\hat{u}_n)\geq I_{\infty}(\hat{u})\geq I(\hat{u})\geq \gamma(c).
\end{equation}
Thus $\gamma(c)$ is achieved.
\end{proof}

Now, we can prove Theorem \ref{1 th}.

{\bf The proof of Theorem \ref{1 th}.}
For any $c\in\mathcal{U}$, by Lemma \ref{1 achieved} there exists $\bar{u}_c \in S(c)$  such that
$I(u_c) = \gamma(c)$ and $I|'_{S(c)}(u_c)=0$. In view of the Lagrange multiplier theorem, there exists $\lambda_c\in\R$ such that\begin{equation}\nonumber
	I'(u_c)+\lambda_cu_c=0.
\end{equation}
Therefore, $(\hat{u}_c, \lambda_c)$ is a solution of $\eqref{1.1.0}$.

\vskip4mm
{\subsection{ $p\in(2,3)$}}
In the subsection, for any $c\in (0,\infty)$ we prove the strict subadditivity inequality \eqref{1 20231008-e3} and $\gamma(c)<0$ in case $p\in (2,3)$.

\begin{lemma}\label{1.2 less0}
 $0>\gamma(c)>-\infty$ for any $c\in (0,\infty)$.
\end{lemma}
\begin{proof}
	Fix a $u\in S(c)$, $\gamma(c)>-\infty$ for any $c>0$ by Lemma \ref{1 well defined}. Now we show $0>\gamma(c)$. Choose a family $u_t(x)= t^2u(tx)$ such that $\|u_t\|_2^2 =tc$. By $2p-3>1$ and (A1), we have
\begin{equation}\nonumber
		I(u_t)=t^3\left[\dfrac{1}{2}\|\nabla u\|^2_2+\dfrac{1}{4}\int_{\R^3}\int_{\R^3}\frac{|u(x)|^2|u(y)|^2}{|x-y|}dxdy
		-\dfrac{1}{p}t^{2p-3}\int_{\R^3}A(t^{-1}x)|u|^pdx\right]< 0
	\end{equation}
	for sufficiently small $t>0$.  This shows that there exists a $t'$ such that
 \begin{equation}\nonumber
		\gamma(s)<0 \  \forall s\in (0,t'].
	\end{equation}
Combining Lemma \ref{1 lesubadditivity} and $\gamma(c)\le 0$ (Lemma \ref{1 well defined}), we deduce $0>\gamma(c)>-\infty$ for all $c>0$.
\end{proof}

\begin{lemma}\label{1.2 limit of 0}
$\lim\limits_{c\to 0}\dfrac{\gamma(c)}{c}=0$.
\end{lemma}
\begin{proof}
Under slightly general conditions about $A(\cdot)$, the existence of normalized (ground state) solutions for the nonautonomous Schr\"{o}dinger equation
\begin{equation}\label{A-equation}
	-\triangle u+\lambda u -A(x)|u|^{p-1}=0,~~p\in(2,\frac{10}{3})
\end{equation}
has been proved in \cite{sx2020nonautonomous}. Hence,
\begin{equation}\nonumber
	\tau(c) := \inf\limits_{u\in S(c)} F(u)
\end{equation}
is achieved at some $u_c$,
where $$F(u):=\frac{1}{2}\|\nabla u\|^2_2-\frac{1}{p}\int_{\R^3}A(x)|u|^pdx$$ is the energy functional  of \eqref{A-equation}.
By Lagrange multiplier theorem, the minimizer $u_c$ for $\tau(c)$ solves the equation $\eqref{A-equation}$, and $\lambda=\lambda_c $ is the Lagrange multiplier associated to the minimizer.
By calculation, we have
	\begin{equation}\label{A1 equt2.1}
	-\dfrac{\lambda_c}{2}=\dfrac{\|\nabla u_c \|^2_2-\int_{\R^3}A(x)|u_c|^pdx}{2\|u_c\|^2_2}\le  \dfrac{\frac{1}{2}\|\nabla u_c\|^2_2-\frac{1}{p}\int_{\R^3}A(x)|u_c|^pdx}{\|u_c\|^2_2}=\dfrac{F(u_c)}{c}=\dfrac{\tau(c)}{c}<0.
	\end{equation}
 In order to prove $\lim\limits_{c\to 0}\dfrac{\tau(c)}{c}=0$, it is enough to show that $\lim\limits_{c\to 0}\lambda_c=0$.
Suppose it is false, there exists a sequence $c_n\to 0$ such that $\lambda_{c_n}>M$ for some $M\in (0,1)$. Let $u_n$ be the minimizer of $\tau(c_n)$, we have
	\begin{equation}\nonumber
		\begin{aligned}
			M\|u_n\|_{H^1}^2 &\le \int_{\R^3}|\nabla u_n|^2dx +M\int_{\R^3}|u_n|^2dx \\
			&\le \|\nabla u_c\|^2_2+\lambda_{c_n}\int_{\R^3}|u_n|^2dx=\int_{\R^3}A(x)|u_n|^pdx \\
			&\le \|A\|_{\infty}\|u_n\|^p_p\leq C\|u_n\|^p_{H^1},
		\end{aligned}
	\end{equation}
	which implies that there exists $C'>0$ such that $\|\nabla u_n\|_2^2>C'>0$.
 However,
	\begin{equation}\nonumber
		0>\tau(c_n)=F(u_n)\geq \dfrac{1}{2}\|\nabla u_n\|_2^2-o(1), \text{ as }n\to\infty.
	\end{equation}
	This is a contradiction. Thus
	\begin{equation}\label{A1 equ2.2}
		\lim\limits_{c\to 0}\dfrac{\tau(c)}{c}=0.
	\end{equation}
		Notice that
	$
		\frac{\tau (c)}{c} \le \frac{\gamma (c)}{c} \leq0,
	$
	then by \eqref{A1 equ2.2} we complete the proof.
\end{proof}

Next, we prove that any critical point of $I(u)$ restricted to $S(c)$ satisfies the identity $P(u)\equiv0$, where
\begin{equation}\nonumber
	\begin{aligned}
		 P(u):=B(u)+\frac{1}{4}C(u)-\frac{1}{p}\int_{\R^3}\left(\frac{3}{2}\left(p-2\right)A(x)-\nabla A(x)\cdot x\right)|u|^pdx,
	\end{aligned}
\end{equation}
and
\begin{equation*}
	B(u)=\int_{\R^3}|\nabla u|^2dx,\quad C(u)=\int_{\R^3}\int_{\R^3}\frac{|u(x)|^2|u(y)|^2}{|x-y|}dxdy.
\end{equation*}
\begin{lemma}\label{1.2 pohozaev}
	If $u$ is a critical point of $I(u)|_{S(c)}$, then $P(u)\equiv0$.
\end{lemma}
\begin{proof} Since $u$ is a critical point of $I(u)$ restricted to $S(c)$, there exists a Lagrange multiplier $\lambda\in \R,$ such that
\begin{equation*}
	I^{\prime}(u)+\lambda u\equiv0.
\end{equation*}
Thus $u$ solves equation \eqref{1.1.0}.
For convenience, we denote
$D(u)=\int_{\R^3}|u|^2dx,$
and define
\begin{equation}\nonumber
	F_\lambda(u):=\langle S'_\lambda(u),u\rangle=B(u)+\lambda D(u)+C(u)-\int_{\R^3}A(x)|u(x)|^pdx,
\end{equation}
and
\begin{equation}\nonumber
	P_\lambda(u):=\frac{1}{2}B(u)+\frac{3}{2}\lambda D(u)+\frac{5}{4}C(u)-\frac{1}{p}\int_{\R^3}\left[3A(x)+\nabla A(x)\cdot x\right]|u|^pdx.
\end{equation}
Here, $S_{\lambda}(u)$ is the energy functional corresponding to equation \eqref{1.1.0}
\begin{equation}\nonumber
	S_{\lambda}(u)=\frac{1}{2}B(u)+\frac{1}{4}C(u)+\frac{\lambda}{2}D(u)-\frac{1}{p}\int_{\R^3}A(x)|u|^pdx.
\end{equation}
Clearly, $F_{\lambda}(u)\equiv0$, $S_\lambda(u) = I(u) + \frac{\lambda}{2}D(u) $ and
\begin{equation}\label{2 equP.1}
	\frac{3}{2}F_{\lambda}(u)-P_{\lambda}(u)=P(u).
\end{equation}
The conclusion $P(u)\equiv0$ is follows by \eqref{2 equP.1} and the following Pohozaev identity for the Schr$\ddot{o}$dinger-Poisson equation \eqref{1.1.0}:
$$P_{\lambda}(u)\equiv0~~\text{for~any~solution~}u~\text{of}~\eqref{1.1.0}.$$
Now all that remains is to prove the above Pohozaev identity. Take $\varphi\in C_{c}^{1}(\R^3)$ such that $\varphi=1$ on $B_1$. Testing equation \eqref{1.1.0} by $v_s$ with
$$v_s(x)=\varphi(sx)x\cdot \nabla u(x),~s>0,~x\in \R^3.$$
Then
\begin{equation}\label{eq:20220425-e2}
\int_{\R^3}\nabla u\cdot \nabla v_s dx+\lambda \int_{\R^3}u v_sdx +\int_{\R^3}(I_2\star |u|^2)uv_sdx =\int_{\R^3}A(x)|u|^{p-2}uv_sdx.
\end{equation}
According to the regularity theory of elliptic equations, we have $u\in H^2_{loc}(\R^3)$. Then by a direct computation,
\begin{equation}\label{eq:20220425-e3}
\begin{cases}
\lim_{s\rightarrow 0}\int_{\R^3}u v_sdx=-\frac{3}{2}\|u\|_2^2,\\
\lim_{s\rightarrow 0}\int_{\R^3}\nabla u\nabla v_sdx=-\frac{1}{2}\|\nabla u\|_2^2,\\
\lim_{s\rightarrow 0}\int_{\R^3}\int_{\R^3}(I_2\star |u|^2)uv_sdx=-\frac{5}{4}\int_{\R^3}(I_2\star |u|^2)|u|^2dx,\\
\lim_{s\rightarrow 0}\int_{\R^3}A(x)|u|^{p-2}uv_sdx=-\frac{1}{p}\int_{\R^3}\left[3A(x)+\nabla A(x)\cdot x\right]|u|^pdx.
\end{cases}
\end{equation}
Hence, by letting $s\rightarrow 0$ in \eqref{eq:20220425-e2}, we obtain $P_{\lambda}(u)\equiv0$.
\end{proof}

Now, we show that $\gamma(c)$ is achieved for a single point $c=c_0$.
\begin{lemma}\label{1.2 point achieved}
There exists a $c_0>0$ such that $\gamma(c_0)$ is achieved.
\end{lemma}
\begin{proof}
	Fix a $c>0$, $\zeta:=\min\{s\in[0,c], \dfrac{\gamma(s)}{s}<0\}$ is well defined by  Lemma \ref{1 lesubadditivity}, Lemma \ref{1 continuous} and Lemma \ref{1.2 less0}. Then, $c_0:=\min\{s\in [0,c], \dfrac{\gamma(s)}{s}=\zeta \}$ is also well defined.
	By Lemma \ref{1.2 less0}, we see that $c_0>0$ and
\begin{equation}\nonumber
		\dfrac{\gamma(c_0)}{c_0} < \dfrac{\gamma(s)}{s}, \ \forall s\in [0,c_0).
	\end{equation}
Thus
\begin{equation}\nonumber
	\gamma(c_0)<\gamma(\alpha) +\gamma(c_0-\alpha), \;\forall \; \alpha>0 .
\end{equation}
With the same argument of Lemma \ref{1 achieved}, $\gamma(c_0)$ is achieved.
\end{proof}

Next, by Lemma \ref{1.2 point achieved} we prove the strong subadditivity inequality.
\begin{lemma}\label{1.2 stronglysubadditivity}
$\gamma (c+\tilde{c}) < \gamma(c) +\gamma(\tilde{c}) $ for any $c,\tilde{c}>0$.
\end{lemma}
\begin{proof}
For every $c>0$, set $\zeta:=\min\{s\in[0,c], \dfrac{\gamma(s)}{s}<0\}$ and $c':=\min\{s\in [0,c] , \dfrac{\gamma(s)}{s}=\zeta \}$. By Lemma \ref{1.2 point achieved}, $c'>0$ and $\dfrac{\gamma(s)}{s} >\dfrac{\gamma(c')}{c'}$ for all $s\in [0,c')$, there exists a $\bar{u} \in S(c')$ such that $I(\bar{u})=\gamma(c')$.\par
	 Now, we prove that $c'$ is exactly $c$. We argue by contradiction assuming that $c'<c$ . Let $\bar{u}^{\theta}=\theta^{-\frac{1}{2}}\bar{u}(\dfrac{x}{\theta})$, we have $\|\bar{u}^{\theta}\|^2_2 = \theta^2\|\bar{u}\|^2_2$. By the definition of $c'$, there exists $\epsilon>0$ such that
	 \begin{equation}\nonumber
	 	\dfrac{\gamma(c')}{c'}\le \dfrac{\gamma(\|\bar{u}^{\theta}\|^2_2)}{\|\bar{u}^{\theta}\|^2_2} \text{ for all } \theta\in (1-\epsilon,1+\epsilon).
	 \end{equation}
Hence,
\begin{equation}\label{1.2 equ14.1}
	\dfrac{I(\bar{u}^{\theta})}{\theta^2 c'}=\dfrac{I(\bar{u}^{\theta})}{\|\bar{u}^{\theta}\|^2_2} \geq\dfrac{\gamma(\|\bar{u}^{\theta}\|^2_2)}{\|\bar{u}^{\theta}\|^2_2} \geq \dfrac{\gamma(c')}{c'} =\dfrac{I(\bar{u})}{c'}.
\end{equation}
Define
 \begin{equation}\nonumber
 	\begin{aligned}
 		K_{\bar{u}}(\theta):=&I(\bar{u}^{\theta}) -\theta^2I(\bar{u})\\
 		=&\frac{1}{2}\left(1-\theta^2\right)\|\nabla \bar{u}\|^2_2+\dfrac{1}{4}(\theta^3-\theta^2)\int_{\R^3}\int_{\R^3}\frac{|\bar{u}(x)|^2|\bar{u}(y)|^2}{|x-y|}dxdy\\
 		&-\dfrac{1}{p}\int_{\R^3}[\theta^{(3-\frac{p}{2})}A(\theta x)-\theta^2A(x)]|\bar{u}|^pdx.
 	\end{aligned}
 \end{equation}
By \eqref{1.2 equ14.1}, $K_{\bar{u}}(\theta)\geq0$ for $\theta\in B_{\epsilon}(1)$.
Moreover, by \eqref{1.2 equ14.1} we see that $K_{\bar{u}}(\theta)$ minimizes at point $\theta=1$, and thus
\begin{equation}\label{20231126-e2}
 	\begin{aligned}
 		0=\dfrac{d}{d \theta} K_{\bar{u}}(1)=&-\|\nabla\bar{u}\|^2_2+\dfrac{1}{4}\int_{\R^3}\int_{\R^3}\frac{|\bar{u}(x)|^2|\bar{u}(y)|^2}{|x-y|}dxdy\\
 		&-\dfrac{1}{p}\int_{\R^3}\left[\left(1-\frac{p}{2}\right)A(x)+\nabla A(x)\cdot x\right]|\bar{u}|^pdx.
 	\end{aligned}
 \end{equation}
Then by \eqref{20231126-e2} and Lemma \ref{1.2 pohozaev}, we derive at
\begin{equation}\label{1.2 equ14.3}
	\left\{\begin{aligned}
		&\dfrac{1}{4}\int_{\R^3}\int_{\R^3}\frac{|\bar{u}(x)|^2|\bar{u}(y)|^2}{|x-y|}dxdy=\dfrac{1}{p}\int_{\R^3}\frac{p-2}{2}A(x)|\bar{u}|^pdx,\\
		&\frac{1}{2}\|\nabla\bar{u}\|^2_2=\dfrac{1}{p}\int_{\R^3}\left[\frac{p-2}{2}A(x)-\frac{1}{2}\nabla A(x)\cdot x\right]|\bar{u}|dx.
	\end{aligned}\right.
\end{equation}
By (A1), (A2), \eqref{1.2 equ14.3} and Lemma \ref{1 lem3.1}, we get
\begin{equation}\label{20231126-e1}
\begin{aligned}
		I(\bar{u})&=\dfrac{1}{2p}\int_{\R^3}[(2p-3)A(x)-\nabla A(x) \cdot x]|\bar{u}|^pdx\\
		&=\frac{1}{p}\int_{\R^3}\left[\frac{3}{2}(p-2)A(x)-\nabla A(x) \cdot x+\frac{p}{2}A(x)\right]|u|^pdx
		&\geq C(A_{\infty},p)\|\bar{u}\|^p_p,
\end{aligned}
\end{equation}
where $C(A_{\infty},p)>0$ is a constant which depends only on $A_{\infty}$ and $p$. We infer from $\|\bar{u}\|^2_2=c'>0$ and \eqref{20231126-e1} that $\gamma(c')=I(\bar{u})>0$, which contradicts with Lemma \ref{1.2 less0}. Thus, $c'=c$. Therefore, for any $c\in (0,\infty)$, $\dfrac{\gamma(s)}{s}$ takes unique minimum value at point $s=c$ on the interval $(0,c]$. Thus, $\dfrac{\gamma(c)}{c}$ is strictly monotonically decreasing respect to $c$. Therefore, by \begin{equation*}
	\gamma(c+\tilde{c})=\dfrac{c}{c+\tilde{c}}\gamma(c+\tilde{c})+\dfrac{\tilde{c}}{c+\tilde{c}}\gamma(c+\tilde{c}) <\gamma(c)+\gamma(\tilde{c})\quad  \forall c, \tilde{c}>0,
\end{equation*} we obtain the strictly subadditivity inequality for the case $p\in (2,3)$.
\end{proof}

\vskip4mm
{\subsection{$p\in(3,\frac{10}{3})$}}
In the subsection, we prove the strict subadditivity inequality \eqref{1 20231008-e3} and $\gamma(c)<0$ for $p\in (3,\frac{10}{3})$.
To this, for $p\in \left(3,\frac{10}{3}\right)$, we observe that $\gamma(c)=0$ for small $c$ and  $\gamma(c)<0$ for large $c$. So we can define
\begin{equation}\label{1.1 equ bar(c)}
	\bar{c}=\inf\left\{c>0\::\:\gamma(c)<0\right\}.
\end{equation}
For $c\geq\bar{c}$, we can verify that $\gamma(c)$ satisfies strict subadditivity and thus $\gamma(c)$ possesses a minimizer. For $0<c<\bar{c}$, we will prove that $\gamma(c)$ has no minimizer.\par
Now, we show the behavior of $\gamma(c)$.
\begin{lemma}\label{1.1 bar(c)behavior}
Define $\bar{c}$ as \eqref{1.1 equ bar(c)}, then $\bar{c}\in (0,\infty)$. Moreover, $\gamma(c)=0$ for $c\in \left(0,\bar{c}\ \right]$, and $\gamma(c)<0$ for $c\in (\bar{c},\infty)$.
	\end{lemma}
	\begin{proof}
		First, we prove that $\bar{c} \in (0,\infty)$. On the one hand, by Gagliardo-Nirenberg's inequality,
		\begin{equation}\label{1.1 equ2.1 G-N-S equation}
			\int_{\R^3}A(x)|u|^pdx\leq \|A\|_\infty\int_{\R^3}|u|^pdx\leq C\|\nabla u\|^{\frac{3}{2}\left(p-2\right)}_2\|u\|^{\frac{6-p}{2}}_2 \;\forall u\in H^1(\R^3).
		\end{equation}
		Then, for any $u\in S(c)$, we have
		\begin{equation}\nonumber
			I(u) \geq \|\nabla u\|^2_2\left(\frac{1}{2}-C\|\nabla u\|_2^{\frac{3}{2}\left(p-2\right)-2}\|u\|_2^{\frac{6-p}{2}}\right).
		\end{equation}
		Thus, $I(u)>0$ for all $u\in S(c)$ when $c$ is sufficiently small. This shows that $\bar{c}>0$.\par
		
		On the other hand, for any fixed $c>0$, for any $u_1 \in S(c)$, let $u_1^{\theta}=\theta^2u_1(\theta x)$, we have $u_1^{\theta} \in S(\theta c)$ and
\begin{equation}\nonumber
\begin{aligned}
     I(u_1^{\theta})&=\theta^3\left(\frac{1}{2}\|\nabla u_1\|^2_2+\frac{1}{4}\int_{\R^3}\int_{\R^3}\frac{|u_1(x)|^2|u_1(y)|^2}{|x-y|}dxdy-\frac{1}{p}\theta^{2p-6}\int_{\R^3}A(\theta^{-1}x)|u_1|^pdx\right)\\
     &\le\theta^3\left(\frac{1}{2}\|\nabla u_1\|^2_2+\frac{1}{4}\int_{\R^3}\int_{\R^3}\frac{|u_1(x)|^2|u_1(y)|^2}{|x-y|}dxdy-\frac{A_{\infty}}{p}\theta^{2p-6}\int_{\R^3}|u_1|^pdx\right).
\end{aligned}
\end{equation}
	By $2p-6>0$, we see that $I(u_1^{\theta})<0$ for $\theta$ large enough. Then $\gamma(\theta c)<0$ for $\theta$ large enough, and $\bar{c}$ is well defined.

		Next we shows that if $c\in (\bar{c},+\infty)$ then $\gamma(c)<0$.
		Fix a $u \in S(c)$  and let $u^{\theta}(x)=\theta^{2}u(\theta x)$, by simple calculations we have $\|u^{\theta}\|^2_2 = \theta c $ and
		\begin{equation*}
			I(u^{\theta})= \dfrac{\theta^3}{2}B(u)+\dfrac{\theta^3}{4}C(u) -\dfrac{\theta^{2p-3}}{p}\int_{\R^3}A(\theta^{-1}x)|u|^pdx.
		\end{equation*}
		For $2p-3 <3$, we have $I(u^{\theta}) <0 $ for sufficiently small $\theta >0$. Then by Lemma \ref{1 lesubadditivity} and Lemma \ref{1 well defined}, we get $\gamma(c)<0$. Using the continuity of $c\mapsto \gamma(c)$, we obtain $\gamma(\bar{c})=0$. This completes the proof.
	\end{proof}

\begin{lemma}\label{1.1 bar(c)>c no achieved}
$\gamma(c)$ has no minimizer for any $0<c< \bar{c}.$
\end{lemma}
\begin{proof}
Suppose the thesis is false. There exist $c\in(0,\bar{c})$ and $u_c\in S(c)$ such that $\gamma(c)=I(u_c)$. By Lemma \ref{1 lem3.1}, we obtain
\begin{equation}\label{1.1 equ8.3}
	\begin{aligned}
		\gamma(tc)&\leq I(\left(u_{c}\right)_t)=t^3\left(\dfrac{1}{2}\|\nabla u_{c}\|^2_2+\dfrac{1}{4}\int_{\R^3}\int_{\R^3}\frac{|u_{c}(x)|^2|u_{c}(y)|^2}{|x-y|}dxdy
		-\dfrac{1}{p}t^{2p-6}\int_{\R^3}A(t^{-1}x)|u_{c}|^p\right)\\
		&\leq t^3\left(\dfrac{1}{2}\|\nabla u_{c}\|^2_2+\dfrac{1}{4}\int_{\R^3}\int_{\R^3}\frac{|u_{c}(x)|^2|u_{c}(y)|^2}{|x-y|}dxdy-\dfrac{1}{p}t^{2p-6}\int_{\R^3}A(x)|u_{c}|^p\right)\\
		&<t^3I(u_{c})=t^3\gamma(c).
	\end{aligned}
\end{equation}
Thus $$\gamma(tc)< t^3\gamma(c)\leq t\gamma(c) \text{ } \forall t>1.$$
	Let $c<c_1<\bar{c}$ and $t=\frac{c_1}{c}>1$, we have
	$$\gamma(c_1)<\frac{c_1}{c}\gamma(c)=0$$
and thus $\gamma(c)<0$. This contradicts with definition of $\bar{c}$.
\end{proof}

\begin{lemma} \label{1.1 subadditivity}
The function $c\mapsto \frac{\gamma(c)}{c}$ is nonincreasing on $(0,\infty)$. In particular, fixed a $c'>\bar{c}$ and suppose that $\gamma(c')$ is achieved, then $$\frac{\gamma(c')}{c'}<\frac{\gamma(c)}{c}~~\forall c \in (0,c'),$$
and equivalently
$$\gamma (c') < \gamma(c) +\gamma(c'-c) ~~\forall c \in (0,c').$$
\end{lemma}
\begin{proof}
For all  $c>0$, there exists $\{u_n \} \subset S(c)$ such that $I(u_n) < \gamma (c) +\dfrac{1}{n}$. Let $t>1$ and $ (u_n)_t(x)= t^2u_n(tx)$, by simple calculations we have $\|(u_n)_t\|_2^2 =tc$ for all $t>1$. Then by Lemma \ref{1 lem3.1} and $3 < p < \frac{10}{3}$ we get
	\begin{equation}\nonumber
		\begin{aligned}
			\gamma(c)&\le I((u_n)_t)=t^3\left(\dfrac{1}{2}\|\nabla u_n\|^2_2+\dfrac{1}{4}\int_{\R^3}\int_{\R^3}\frac{|u_n(x)|^2|u_n(y)|^2}{|x-y|}dxdy-\dfrac{1}{p}t^{2p-6}\int_{\R^3}A(t^{-1}x)|u_n|^p\right)\\
			&< t^3\left(\dfrac{1}{2}\|\nabla u_n\|^2_2+\dfrac{1}{4}\int_{\R^3}\int_{\R^3}\frac{|u_n(x)|^2|u_n(y)|^2}{|x-y|}dxdy-\dfrac{1}{p}t^{2p-6}\int_{\R^3}A(x)|u_n|^p\right)\\
			&<t^3(I(u_n))< t^3\left(\gamma(c)+\dfrac{1}{n}\right).
		\end{aligned}
	\end{equation}
This implies that
\begin{equation}\nonumber
	 \gamma(tc) \le t^3 \gamma(c)\le t\gamma(c),\ \forall t>1,
\end{equation}
and thus
\begin{equation}\nonumber
	\dfrac{\gamma(tc)}{tc} \le \dfrac{\gamma(c)}{c}.
\end{equation}
If $\gamma(c')$ is achieved, i.e. there exists a $u_{c'}$ such that $I(u_{c'})=\gamma(c')$. Let $\left(u_{c'}\right)_t(x)=t^2u_{c'}(tx)$, by the same argument of \eqref{1.1 equ8.3} we obtain
\begin{equation}\label{1.1 equ8.31}
	\begin{aligned}
		\gamma(tc')<t^3I(u_{c'})=t^3\gamma(c').
	\end{aligned}
\end{equation}
Since $\gamma(c')<0$ for every $c'>\bar{c}$, then $$\gamma(tc')<t\gamma(c'),\ \forall t>1.$$ Therefore, $c\mapsto \frac{\gamma(c)}{c}$ is nonincreasing on $(0,\infty)$.

Moreover, for any $c\in (0,c')$, by \eqref{1.1 equ8.31} we have
$$\dfrac{\gamma(c')}{c'}<\dfrac{\gamma(c)}{c}.$$
Thus
\begin{equation}\nonumber
	\gamma(c')=\dfrac{c}{c'}\gamma(c')+\dfrac{c'-c}{c'}\gamma(c') <\gamma(c)+\gamma(c'-c)\quad  \forall c\in (0,c'),
\end{equation}
which completes the proof.
\end{proof}

 Next, we prove that $\gamma(c)$ is achieved for the end point case: $c=\bar{c}$.

\begin{lemma}\label{1.1 bar{c} achieved}
 $\gamma(\bar{c})$ is achieved.
\end{lemma}
\begin{proof}
	By the definition of $\bar{c}$ and Lemma \ref{1 continuous}, we have $\gamma(\bar{c})=0$. Let $c_n=\bar{c}+\frac{1}{n}$, by Lemma \ref{1 continuous} we have $\gamma(c_n)\to \gamma(\bar{c})$. Furthermore, by Lemma \ref{1 achieved}, we know that
	$\gamma(c_n)<0$ and $\gamma(c_n)$ possesses a minimizer $u_n$ for every $n\in \N$.
	With the same argument of Lemma $3.4$ in \cite{nonexistence}, we have $\int_{\R^3}|u_n|^pdx\not\to0$. Then
	by Lemma I.1 in \cite{lionI1984}, $\{u_n\}$ doesn't vanish. Namely, there exist a constant $\delta>0$ and a sequence $\{y_n\}\subset \R^3$ such that
	\begin{equation}\nonumber
   \int_{B_1(0)}|u_n(x+y_n)|^2dx >\delta>0.
	\end{equation}
 Let $v_n(x)=u_n(x+y_n)$. Clearly, $\left\{v_n\right\}$ is bounded in $H^1(\R^3)$ and there exists $v_0\in H^1(\R^3)$ such that
  \begin{equation}\nonumber
  	v_n\rightharpoonup v_0\text{ in } H^1(\R^3),\; v_n\to v_0\text{ in } L_{loc}^2(\R^3).
  \end{equation}
We see from $$\int_{B_1(0)}|v_0|dx \geq\lim\limits_{n\to \infty} \int_{B_1(0)}|v_n|dx>\frac{\delta}{2}>0$$ that $v_0\neq 0$.\par
Now we show that $v_0$ is a minimizer of $\gamma(\bar{c})$. Combining Lemma \ref{1 le 2.1} and Lemma \ref{1 continuous}, we deduce that
\begin{equation}\label{20231130-e3}
	\begin{aligned}
		0&=\gamma(\bar{c})=\lim\limits_{n\to \infty}\gamma(c_n)=\lim\limits_{n\to \infty}I(v_n)\\
		&=I(v_0)+\lim\limits_{n\to \infty}I(v_n-v_0) \geq I(v_0)+\lim\limits_{n\to \infty}\gamma(\|v_n-v_0\|^2_2)\\
		&=I(v_0)+\gamma(\bar{c}-\|v_0\|^2_2)\geq \gamma(\|v_0\|_2^2)+\gamma(\bar{c}-\|v_0\|^2_2).
	\end{aligned}
\end{equation}
Notice that by the weak lower semicontinuity of norms
$$\|v_0\|^2_2\leq\liminf\limits_{n\to\infty}\|v_n\|_2^2=\liminf\limits_{n\to\infty}c_n= \bar{c}.$$
Hence, by Lemma \ref{1.1 bar(c)behavior} and \eqref{20231130-e3}, we obtain $$\gamma(\|v_0\|^2_2)=\gamma(\bar{c}-\|v_0\|^2_2)=0=I(v_0),$$
and thus $\gamma(\|v_0\|^2_2)$ is achieved.
And by Lemma \ref{1.1 bar(c)>c no achieved}, $\|v_0\|^2_2=\bar{c}$ must hold. This completes the proof.

\end{proof}

\vskip4mm
{\section{$L^2$ -supercritical case $p\in(\frac{10}{3},6)$}}
\setcounter{equation}{0}
In this section, we prove Theorem \ref{2 th} under $p\in(\frac{10}{3},6)$ and (A1)-(A3).

\begin{lemma}\label{2 I(u)>I(u^t)}
Let $u\in \mathcal{P}(c)$, then $I(u)>I(u^t)$ for all $t\in (0,1)\cup(1,\infty)$, where $u^t(x)=t^{\frac{3}{2}}u(tx)$.
\end{lemma}
\begin{proof}
	Fix a $u\in \mathcal{P}(c)$, let
\begin{equation}\nonumber
\begin{aligned}
	g(x,u,t):=&I(u)-I(u^t)\\
	=&\frac{1}{2}(1-t^2)\|\nabla u\|^2_2+\frac{1}{4}(1-t)\int_{\R^3}\int_{\R^3}\frac{|u(x)|^2|u(y)|^2}{|x-y|}dxdy
	+\frac{1}{p}t^{\frac{3}{2}(p-2)}\int_{\R^3}A(t^{-1}x)|u(x)|^pdx\\
	&-\frac{1}{p}\int_{\R^3}A(x)|u(x)|^pdx.
\end{aligned}
\end{equation}
Then
\begin{equation}\nonumber
	\begin{aligned}
		\dfrac{\partial }{\partial t}g(x,u,t)=&-t\|\nabla u\|^2_2-\frac{1}{4}\int_{\R^3}\int_{\R^3}\frac{|u(x)|^2|u(y)|^2}{|x-y|}dxdy\\
		&+\frac{1}{p}t^{\frac{3}{2}(p-2)-1}\int_{\R^3}[\frac{3}{2}A(t^{-1}x)-\nabla A(t^{-1}x)\cdot (t^{-1}x)]|u(x)|^pdx\\
		=&-t\|\nabla u\|^2_2-\frac{1}{4}\int_{\R^3}\int_{\R^3}\frac{|u(x)|^2|u(y)|^2}{|x-y|}dxdy
		+th(x,u,t)\\
		=&-\frac{1}{4}\int_{\R^3}\int_{\R^3}\frac{|u(x)|^2|u(y)|^2}{|x-y|}dxdy
		+t[h(x,u,t)-\|\nabla u\|^2_2],
	\end{aligned}
\end{equation}
where $$h(x,u,t)=\frac{1}{p}t^{\frac{3}{2}(p-2)-2}\int_{\R^3}[\frac{3}{2}A(t^{-1}x)-\nabla A(t^{-1}x)\cdot (t^{-1}x)]|u(x)|^pdx.$$
Thanks to (A1), (A2) and Lemma \ref{1 lem3.1}, we have that
 \begin{equation}\label{2 equ2.1}
	h(x,u,t)\text{ is strictly monotone increasing on } t\in(0,\infty) \text{ for every } x\in \R^3.
\end{equation}
Since $\|u\|^2_2>0$ and $\frac{3}{2}(p-2)\in(2,6)$, we see that
\begin{equation}\label{2 equ2.2}
  \lim\limits_{t\to 0}h(x,u,t)=0, \; \lim\limits_{t\to +\infty}h(x,u,t)=+\infty.
\end{equation}
This and \eqref{2 equ2.1} imply that there is a unique $t_0\in (0,\infty)$ such that $h(x,u,t_0)=\|\nabla u\|^2_2>0$. Thus,
\begin{equation}\label{20231126-e3}
t[h(x,u,t)-\|\nabla u\|^2_2]~\text{decreases~on}~(0,t_0)~\text{and~increases~on}~(t_0,\infty).
\end{equation}
Moreover, by \eqref{2 equ2.2} we deduce that
 \begin{equation}\nonumber
 	\lim\limits_{t\to 0}t[h(x,u,t)-\|\nabla u\|^2_2]=0, \; \lim\limits_{t\to +\infty}t[h(x,u,t)-\|\nabla u\|^2_2]=+\infty.
 \end{equation}
Then by \eqref{20231126-e3}, there exists a unique $t'\in (t_0,+\infty)$ such that $\dfrac{\partial}{\partial t}g(x,u,t)|_{t=t'}=0$.\par
One the other hand, clearly $g(x,u,1)=0$ for all $x\in \R^3$. Then by $\frac{3}{2}(p-2)\in (2,6)$ and \eqref{A2 equ1.3} ( Lemma \ref{1 lem3.1}), we have
\begin{equation}\nonumber
\begin{aligned}
		I(u)&=I(u)-\frac{1}{2}P(u)\\
		&=\frac{1}{8}\int_{\R^3}\int_{\R^3}\frac{|u(x)|^2|u(y)|^2}{|x-y|}dxdy
		+\frac{1}{2p}\int_{\R^3}\left[\left(\frac{3}{2}(p-2)-2\right)A(x)-\nabla A(x)\cdot x\right]|u|^pdx>0.
\end{aligned}
\end{equation}
Thus we deduce that
\begin{equation}\nonumber
	\lim\limits_{t\to 0}g(x,u,t)=I(u)>0, \; \lim\limits_{t\to+\infty}g(x,u,t)=+\infty.
\end{equation}
Hence, $t=1$ is the unique minimum point of $g(x,u,\cdot)$, and $g(x,u,t)> g(x,u,1)=0\text{ for all } t\in (0,1)\cup(1,\infty)$.
\end{proof}

\begin{lemma}\label{2 unique t_u}
For any $u\in S(c) $ there exists a unique $t_u>0$ such that $u^{t_u}\in \mathcal{P}(c)$.
\end{lemma}
\begin{proof}
	Fix a $u\in S(c)$ and define a function $\omega(t):=I(u^t)$ on $(0,\infty)$. By simple calculation,
	\begin{equation}\nonumber
		\begin{aligned}
			\dfrac{\partial}{\partial t}\omega(t)&=t\|u\|^2_2+\frac{1}{4}\int_{\R^3}\int_{\R^3}\frac{|u(x)|^2|u(y)|^2}{|x-y|}dxdy
			-th(x,u,t)\\
			&=\frac{1}{4}\int_{\R^3}\int_{\R^3}\frac{|u(x)|^2|u(y)|^2}{|x-y|}dxdy
			-t\left[h(x,u,t)-\|u\|^2_2\right],
		\end{aligned}
	\end{equation}
where \begin{equation}\nonumber
	h(x,u,t)=\frac{1}{p}t^{\frac{3}{2}(p-2)-2}\int_{\R^3}\left[\frac{3}{2}A(t^{-1}x)-\nabla A(t^{-1}x)\cdot (t^{-1}x)\right]|u(x)|^pdx.
\end{equation}
By (A1), (A2) and Lemma \ref{1 lem3.1}, $h(x,u,t)$ is strictly monotone increasing on $t\in (0,\infty)$, and
\begin{equation}\label{20231130-e10}
	\lim\limits_{t\to 0}h(x,u,t)=0, \; \lim\limits_{t\to +\infty}h(x,u,t)=+\infty.
\end{equation}
Then there exists a unique $t_0$ such that $h(x,u,t_0)=\|u\|^2_2$. Thus $t[h(x,u,t)-\|u\|^2_2] $ is strictly increasing on $(0,t_0)$ and is strictly decreasing on $(t_0,+\infty)$. Moreover, by \eqref{20231130-e10} we have
\begin{equation}\nonumber
	\lim\limits_{t\to 0}t\left[h(x,u,t)-\|u\|^2_2\right]=0, \; \lim\limits_{t\to +\infty}t\left[h(x,u,t)-\|u\|^2_2\right]=+\infty.
\end{equation}
Thus there is a unique $t_u\in (t_0,\infty)$ such that $\dfrac{\partial}{\partial t}\omega(t_u)=0$.  \par
On the other hand, since
\begin{equation}\nonumber
\begin{aligned}
	P(u^t)=&t^2\|\nabla u\|^2_2+\frac{t}{4}\int_{\R^3}\int_{\R^3}\frac{|u(x)|^2|u(y)|^2}{|x-y|}dxdy\\
	&-\frac{1}{p}t^{\frac{3}{2}(p-2)-1}\int_{\R^3}\left[\frac{3}{2}(p-2)A(t^{-1}x)-\nabla A(t^{-1}x)\cdot (t^{-1}x)\right]|u|^pdx=t\cdot\dfrac{\partial}{\partial t}\omega(t),
\end{aligned}
\end{equation}
we have
\begin{equation}\nonumber
	P(u^t)=0 \;\Leftrightarrow \; \dfrac{\partial}{\partial t}\omega(t)=0\; \Leftrightarrow \; u^t\in \mathcal{P}(c).
\end{equation}
It is easy to verify that $\lim\limits_{t\to 0}\omega(t)=0$, $\omega(t) > 0$ for $t > 0 $ sufficiently small and $\omega(t) < 0$ for $t$ sufficiently large. Therefor $\max_{t\in [0,\infty)}\omega(t)$ is achieved at some $t_u>0$. From the proof of Lemma \ref{2 I(u)>I(u^t)}, we see that $I(u^t)$ has only one extreme point, so the uniqueness of $t_u$ is proved.
\end{proof}

Combining Lemma \ref{2 I(u)>I(u^t)} and Lemma \ref{2 unique t_u}, we prove
\begin{lemma}\label{2 m(c)defination} 	Let \begin{equation*}
	m(c):=\inf_{u \in \mathcal{P}(c)}I(u).
\end{equation*}
Then
	\begin{equation}\label{m(c) 2 defination}
		\inf_{u \in \mathcal{P}(c)}I(u)=m(c)=\inf_{u \in S(c)}\max_{t>0}I(u^t).
	\end{equation}
\end{lemma}
\begin{proof}
From Lemma \ref{2 I(u)>I(u^t)}, we can find $t'$ such that $I(u^{t'})=\max_{t>0}I(u^t)$ and $\dfrac{\partial}{\partial t}I(u^{t})|_{t=t'}=0$. By calculation, we have
\begin{equation}\nonumber
\begin{aligned}
	P(u^{t'})=&t'^2\|\nabla u\|^2_2+\frac{1}{4}t'\int_{\R^3}\int_{\R^3}\frac{|u(x)|^2|u(y)|^2}{|x-y|}dxdy\\
	&-\frac{1}{p}t'^{\frac{3}{2}(p-2)}\int_{\R^3}\left[\frac{3}{2}(p-2)A(t'^{-1}x)-\nabla A(t'^{-1}x)\cdot (t'^{-1}x)\right]|u|^pdx\\
	=&t'\dfrac{\partial}{\partial t}I(u^{t'})=0.
\end{aligned}
\end{equation}
Thus $u^{t'}\in \mathcal{P}(c)$ and \eqref{m(c) 2 defination} holds.
\end{proof}

Now let's show the behavior of $m(c)$.

\begin{lemma}\label{2 m(c) continuous}
 $m(c)$ is continuous for any $c>0$.
\end{lemma}
\begin{proof}
For all $c>0$, choose ${{c}_{n}}$ and $u_n\in \mathcal{P} \left( {{c}_{n}} \right)$ such that ${{c}_{n}}\to c$ and \begin{equation*}
	 m\left( {{c}_{n}} \right)+\frac{1}{n}>I\left( {{u}_{n}} \right).
\end{equation*}
By directly calculation, we have $\left\| \sqrt{\dfrac{c}{c_n}}u_n \right\|_{2}^{2}=c$, that is $\sqrt{\dfrac{c}{c_n}}u_n\in S\left( c \right)$.
 By Lemma \ref{2 unique t_u}, there exist ${{t}_{n}}>0$ such that $${{\left( \sqrt{\dfrac{c}{{{c}_{n}}}}{{u}_{n}} \right)}^{{{t}_{n}}}}=t_{n}^{\frac{3}{2}}\sqrt{\dfrac{c}{{{c}_{n}}}}{{u}_{n}}\left( {{t}_{n}}x \right)\in \mathcal{P} \left( c\right).$$
Then it follows from Lemma \ref{2 I(u)>I(u^t)} that
\begin{equation}\label{20231126-e5}
\begin{aligned}
	m(c)&\le I\left(\left(\sqrt{\dfrac{c}{c_n}}u_n\right)^{t_n}\right)\le I\left(\sqrt{\dfrac{c}{c_n}}u_n\right)\\
	&=\frac{1}{2}\frac{c}{c_n}\|\nabla u_n\|^2_2+\frac{1}{4}\left(\frac{c}{c_n}\right)^2\int_{\R^3}\int_{\R^3}\frac{|u_n(x)|^2|u_n(y)|^2}{|x-y|}dxdy-\frac{1}{p}\left(\frac{c}{c_n}\right)^{\frac{p}{2}}\int_{\R^3}A(x)|u_n|^pdx\\
	&=I(u_n)+o(1)= m(c_n)+o(1).
	\end{aligned}
\end{equation}
One the other hand, choose $\{v_n\}\subset \gamma(c)$ such that
\begin{equation*}
	m\left(c\right)+\frac{1}{n}>I\left(v_n\right).
\end{equation*}
Then Lemma \ref{2 I(u)>I(u^t)} and Lemma \ref{2 unique t_u} imply that there is $t_n'>0$ such that $$\left(\sqrt{\dfrac{c}{c_n}}v_n\right)^{t_n'}=t_n'^{\frac{3}{2}}\sqrt{\dfrac{c}{c_n}}u_n(t_n'x)\in \mathcal{P}(c_n)$$ and
\begin{equation}\label{20231126-e6}
	m(c_n)\le I\left(\left(\sqrt{\dfrac{c}{c_n}}v_n\right)^{t_n'}\right)\le I\left(\sqrt{\dfrac{c}{c_n}}v_n\right)=I(v_n)+o(1)=m(c)+o(1).
\end{equation}
Combining \eqref{20231126-e5} with \eqref{20231126-e6}, we get $\lim\limits_{n\to \infty}m(c_n)=m(c)$.
\end{proof}

\begin{lemma}\label{2 m(c)>0}
(i) $\forall u\in \mathcal{P}(c)$, there exists $\xi _0=\xi_0(c)>0$ such that $\|\nabla u\|_2\geq \xi_0$;\\
(ii) $m(c)=\inf_{u \in \mathcal{P}(c)}I(u)>0$;\\
(iii) $m(c)\to +\infty$ as $c\to 0$.
\end{lemma}
\begin{proof}
	(i) By Lemma \ref{1 lec(u)} and $u\in \mathcal{P}(c)$,
\begin{equation}\nonumber
	\begin{aligned}
		\|\nabla u\|^2_2&=-\frac{1}{4}\int_{\R^3}\dfrac{|u(x)|^p|u(y)|^p}{|x-y|}dx+\dfrac{1}{p}\int_{\R^3}\left[\frac{3}{2}(p-2)A(x)-\nabla A(x)\cdot x\right]|u(x)|^pdx\\
		&<\frac{1}{64\pi}\|\nabla u\|^2_2+\frac{1}{p}\int_{\R^3}\left[\frac{3}{2}(p-2)A(x)-\nabla A(x)\cdot x\right]|u(x)|^pdx.
	\end{aligned}
\end{equation}
From Lemma \ref{1 lem3.1}, for any $\epsilon>0$, there is $R>0$ such that $|\nabla A(x)\cdot x|< \epsilon$ for $|x|>R$. Then by $A(\cdot)\in C^1$ we have $\|\nabla A(x)\cdot x\|_{L^{\infty}}<\infty$ and
\begin{equation}\nonumber
	\begin{aligned}
		\|\nabla u\|^2_2<
	\frac{1}{64\pi}\|\nabla u\|^2_2+ \frac{1}{p}\left[\dfrac{3}{2}(p-2)\|A\|_{\infty}+\|\nabla A(x)\cdot x\|_{L^{\infty}}\right]\|u\|^p_p.
	\end{aligned}
\end{equation}
Thus, by Gagliardo-Nirenberg inequality
\begin{equation}\nonumber
	\|u\|^p_p \le C_1\|\nabla u\|_2^{\frac{3(p-2)}{2}}\|u\|_2^{\frac{6-p}{2}} \;\;\forall u\in H^1(\R^3),\  p\in(2,2^*),
\end{equation}
we have
\begin{equation}\nonumber
	0<\left(\dfrac{1}{64\pi}-1\right)\|\nabla u\|^2_2+C_2\|\nabla u\|_2^{\frac{3(p-2)}{2}}\|u\|_2^{\frac{6-p}{2}},
\end{equation}
where $C_2=C_1\dfrac{1}{p}\left[\dfrac{3}{2}(p-2)\|A\|_{\infty}+\|\nabla A(x)\cdot x\|_{L^{\infty}}\right]$. Then by $\|\nabla u\|^2_2\neq 0$ and  $\dfrac{3(p-2)}{2}>2$, we get
$$C_2\|\nabla u \|^{\frac{3p-10}{2}}_2\|u\|_2^{\frac{6-p}{2}}>1-\dfrac{1}{64 \pi}.$$
That is $\|\nabla u\|_2>\xi _0$, where $\xi_0=\left(\dfrac{64\pi -1}{64\pi C_2\|u\|_2^{\frac{6-p}{2}}}\right)^{\frac{2}{3p-10}}$.\par
(ii) For any $u\in \mathcal{P}(c)$, by (i) we see that
\begin{equation}\label{2 equ5.3}
\begin{aligned}
	I(u)=&I(u)-\dfrac{2}{3(p-2)}P(u)\\
	=&\left(\frac{1}{2}-\dfrac{2}{3(p-2)}\right)\|\nabla u\|_2^2+\frac{1}{4}\left(1-\dfrac{2}{3(p-2)}\right)\int_{\R^3}\int_{\R^3}\frac{|u(x)|^2|u(y)|^2}{|x-y|}dxdy\\
	&-\dfrac{2}{3p(p-2)}\int_{\R^3}[\nabla A(x)\cdot x]|u(x)|^pdx\geq \left(\frac{1}{2}-\dfrac{2}{3(p-2)}\right)\xi_0,
\end{aligned}
\end{equation}
where $\xi_0=\left(\dfrac{64\pi -1}{64\pi C_2c^{\frac{6-p}{4}}}\right)^{\frac{2}{3p-10}}$.
Therefore, $m(c)=\inf_{u \in \mathcal{P}(c)}I(u)\geq \left(\frac{1}{2}-\dfrac{2}{3(p-2)}\right)\xi_0>0.$
\par
(iii) Let $\{u_n\}\subset\mathcal{P}(c)$ such that $m(c)+\frac{1}{n}>I(u_n)$, by \eqref{2 equ5.3} we have
\begin{equation}\nonumber
	m(c)+\frac{1}{n}> I(u_n)\geq\left(\frac{1}{2}-\dfrac{2}{3(p-2)}\right)\left(\dfrac{64\pi -1}{64\pi C_2\|u_n\|_2^{\frac{6-p}{2}}}\right)^{\frac{2}{3p-10}}.
\end{equation}
Then $I(u_n)\to +\infty$ as $\|u_n\|^2_2\to 0^+$, i.e. $m(c)\to +\infty $ as $c\to 0$.
\end{proof}

Inspired by \cite{sx2020nonautonomous}, we prove the following lemma.

\begin{lemma}\label{2 f(t,u)}
We have
	\begin{equation}\nonumber
		f(t,u):=I(u)-I(u^t)+\dfrac{t^2-1}{2}P(u)\ge 0\ \; \forall t>0,\ u\in H^1(\R^3).
	\end{equation}
In particularly, $f_{\infty}(t,u):=I_{\infty}(u)-I_{\infty}(u^t)+\dfrac{t^2-1}{2}P_{\infty}(u)\geq 0$.
\end{lemma}
\begin{proof}
For all $ u\in H^1(\R^3),$
\begin{equation}\nonumber
\begin{aligned}
f(t,u)=&\frac{1}{2}(1-t^2)\|\nabla u\|^2_2+\frac{1}{4}(1-t)\int_{\R^3}\int_{\R^3}\frac{|u(x)|^2|u(y)|^2}{|x-y|}dxdy\\
&-\frac{1}{p}\int_{\R^3}\left[1-t^{\frac{3}{2}(p-2)}A(t^{-1}x)\right]|u(x)|^pdx+\frac{t^2-1}{2}\|\nabla u\|^2_2+\\
&\frac{t^2-1}{8}\int_{\R^3}\int_{\R^3}\frac{|u(x)|^2|u(y)|^2}{|x-y|}dxdy-\dfrac{1}{p}\frac{t^2-1}{2}\int_{\R^3}\left[\frac{3}{2}(p-2)A(x)-\nabla A(x)\cdot x\right]|u|^pdx\\
=&\frac{t^2-1}{8}\int_{\R^3}\int_{\R^3}\frac{|u(x)|^2|u(y)|^2}{|x-y|}dxdy-\dfrac{1}{p}\frac{t^2-1}{2}\int_{\R^3}\left[\frac{3}{2}(p-2)A(x)-\nabla A(x)\cdot x\right]|u|^pdx\\
&-\frac{1}{p}\int_{\R^3}\left[1-t^{\frac{3}{2}(p-2)}A(t^{-1}x)\right]|u(x)|^pdx.
\end{aligned}
\end{equation}
Take the derivative of $f$ about $t$,
\begin{equation}\nonumber
\begin{aligned}
	\dfrac{\partial}{\partial t}f(t,u)=&\frac{t-1}{4}\int_{\R^3}\int_{\R^3}\frac{|u(x)|^2|u(y)|^2}{|x-y|}dxdy-\dfrac{1}{p}\int_{\R^3}\left[\frac{3}{2}(p-2)A(x)-\nabla A(x)\cdot x\right]|u|^pdx\\
	&+\dfrac{1}{p}t^{\frac{3}{2}(p-2)-1}\int_{\R^3}\left[\frac{3}{2}(p-2)A(t^{-1}x)-\nabla A(t^{-1}x)\cdot t^{-1}x\right]|u|^pdx.
\end{aligned}
\end{equation}
Then by (A2), we see that $f(\cdot,u)$ is strictly decreasing on $(0,1)$ and strictly decreasing on $(1,+\infty)$, i.e. $f(t,u)\geq f(1,u)=0$ for all $t>0$. Since $A(x)\equiv A_\infty$ satisfies (A2), thus we also have $f_{\infty}(t,u)\ge 0$ for all $t>0$.
\end{proof}

From the above behavior of $m(c)$, we deduce that there exists $c'$ such that $m(c')$ is achieved.

\begin{lemma}\label{2 drceasing}
$m(c)$ is non-increasing about $c$. Moreover, if $m(c)$ is achieved, then $m(c)$ is strictly decreasing about $c$.
\end{lemma}
\begin{proof}
    For any $c_2>c_1>0$, there exists $\left\{u_n \right\} \subset \mathcal{P}(c_1)$ such that $I(u_n)< m(c_1)+\frac{1}{n}$. Let $\xi:=\sqrt{\dfrac{c_2}{c_2}}\in (1,\infty)$ and $v_n(x)=\xi^{-5}u_n(\xi^{-4}x)$, we have $\|v_n\|^2_2=\xi^2\|u_n\|^2_2 =c_2$. By Lemma \ref{2 unique t_u}, there exist $t_n$ such that $v_n^{t_n} \in \mathcal{P}(c_2)$. Then it follows from (A3), $\frac{3}{2}-\frac{5}{4}p<0$ and Lemma \ref{2 I(u)>I(u^t)} that
    \begin{equation}\nonumber
      \begin{aligned}
        &m(c_2)\leq I(v_n^{t_n}) \\
        =&\frac{1}{2}t_n^2\xi^{-5}\|\nabla u\|^2_2+\frac{1}{4}t_n\int_{\R^3}\int_{\R^3}\frac{|u_n(x)|^2|u_n(y)|^2}{|x-y|}dxdy
        -\frac{1}{p}t_n^{3-\frac{3}{2}(p-2)}\xi^{12-5p}\int_{\R^3}A\left(\xi^4\frac{x}{t_n}\right)|u_n|^pdx\\
        =&I(u_n^{t_n})+t_n^2(\xi^{-6}-1)\|\nabla u_n\|^2_2
        +\frac{1}{p}t_n^{\frac{3}{2}(p-2)}\int_{\R^3}\left[A\left(\frac{x}{t_n}\right)-\left(\xi^4\right)^{3-\frac{5}{4}p}A\left(\frac{x}{t_n}\xi^4\right)\right]|u_n|^pdx \\
        <&I\left(u_n^{t_n}\right)+\frac{1}{p}t_n^{\frac{3}{2}(p-2)}\int_{\R^3}\left[A\left(\frac{x}{t_n}\right)-\left(\xi^4 \right)^ {\frac{3}{2}}A\left(\frac{x}{t_n}\xi^4\right)\right]\left(\xi^4\right)^{\frac{3}{2}-\frac{5}{4}p}|u_n|^pdx\\
        \leq &I\left(u_n^{t_n}\right)+\frac{1}{p}t_n^{\frac{3}{2}(p-2)}\int_{\R^3}\left[A\left(\frac{x}{t_n}\right)-\left(\xi^4 \right)^{\frac{3}{2}}A\left(\frac{x}{t_n}\xi^4 \right) \right]|u_n|^pdx\\
        \leq &I\left(u_n^{t_n}\right)\leq I(u_n)<m(c_1)+\frac{1}{n}.
        \end{aligned}
    \end{equation}
Letting $n\to\infty$, we obtain $m(c_2)\leq m(c_1)$.\par
Next, we assume that $m(c)$ is achieved, i.e., there exists $\bar{u}\in \mathcal{P}$ such that $I(\bar{u})=m(c)$. For any $c'>c>0$, let $\xi:=\sqrt{\dfrac{c'}{c}}\in (1,\infty)$ and $\bar{v}(x)=\xi^{-5}\bar{u}(\xi^{-4}x)$, we have $\|\bar{v}\|^2_2=\xi^2\|\bar{u}\|^2_2 =c'$. By Lemma \ref{2 unique t_u}, there exists $t_0$ such that $\bar{v}^{t_0} \in \mathcal{P}(c')$. Then it follows from (A3), $\frac{3}{2}-\frac{5}{4}p<0$ and Lemma \ref{2 I(u)>I(u^t)} that
\begin{equation}\nonumber
      \begin{aligned}
        m(c')\leq& I(\bar{v}^{t_0}) \\
        <&I\left(\bar{u}^{t_0}\right)+\frac{1}{p}t_0^{\frac{3}{2}(p-2)}
        \int_{\R^3}\left[A\left(\frac{x}{t_0}\right)-\left(\xi^4 \right)^ {\frac{3}{2}}A\left(\frac{x}{t_0}\xi^4\right)\right]\left(\xi^4\right)^{\frac{3}{2}-\frac{5}{4}p}|\bar{u}|^pdx\\
        \leq& I\left(\bar{u}^{t_0}\right)\leq I(\bar{u})=m(c).
        \end{aligned}
    \end{equation}
This shows that $m(c')<m(c)$.
\end{proof}

By Lemma \ref{2 m(c)defination}, we have
\begin{equation}\label{2 equ m_infty>m}
    m(c)\leq m_{\infty}(c).
\end{equation}
In view of \eqref{2 equ m_infty>m}, we can prove the following lemma.
\begin{lemma}\label{2 achieved}
 $m(c)$ is achieved.
\end{lemma}
\begin{proof}
From Lemma \ref{2 unique t_u} and Lemma \ref{2 m(c)>0}, $\mathcal{P}(c)\neq \emptyset $ and $m(c)>0$. Let $\{u_n\}\subset\mathcal{P}(c)$ be such that $I\left( {{u}_{n}} \right)\to m\left( c \right)$. By \eqref{2 equ5.3} we have
\begin{equation}\nonumber
	m\left( c \right)+o\left( 1 \right)=I\left( {{u}_{n}} \right)\ge \left( \frac{1}{2}-\frac{2}{3\left( p-2 \right)} \right)\|\nabla u_n\|^2_2,
\end{equation}
which implies that $\left\{ u_n \right\}$ is bounded in $H^1(\R^3)$. Up to a subsequence,
\begin{equation}\nonumber
u_n\rightharpoonup \bar{u}\text{ in }H^1(\R^3),\; u_n\to \bar{u}\text{ in }L^s_{loc}(\R^3) \text{ for }s\in [2,2^*),\; u_n\to \bar{u}\text{ a.e. }\R^3.
\end{equation}

Case (i) $\bar{u}\ne 0$. Let ${{w}_{n}}={{u}_{n}}-\bar{u}$ and by Lemma \ref{1 le 2.2}, we get
\begin{equation}\nonumber
	I\left( {{u}_{n}} \right)=I\left( \bar{u} \right)+I\left( {{w}_{n}} \right)+o\left( 1 \right)\; \text{and}\; P\left( {{u}_{n}} \right)=P\left( \bar{u} \right)+P\left( {{w}_{n}} \right)+o\left( 1 \right).
\end{equation}
For any $u \in {{H}^{1}}(\R^3)$, define
\begin{equation}\nonumber
	\begin{aligned}
		\Phi \left( u \right):=&I\left( u \right)-\frac{1}{2}P\left( u \right)=\frac{1}{8}\int_{\R^3}\int_{\R^3}\frac{|u(x)|^2|u(y)|^2}{|x-y|}dxdy-\frac{1}{p}\int_{\R^3}{A\left( x \right){{\left| u \right|}^{p}}}dx\\
		&+\frac{1}{2p}\int_{\R^3}{\left[ \frac{3}{2}\left( p-2 \right)A\left( x \right)-\nabla A\left( x \right)\cdot x \right]{{\left| u \right|}^{p}}}dx.
	\end{aligned}
\end{equation}
From Lemma \ref{1 lem3.1}, we see that
$$\Phi(u) \geq \frac{1}{8}\int_{\R^3}\int_{\R^3}\frac{|u(x)|^2|u(y)|^2}{|x-y|}dxdy +\dfrac{1}{p}\int_{\R^3}\left[\frac{3}{4}(p-2)-1\right]A(x)|u|^pdx.$$
Since $p\in (\frac{10}{3},6)$ and $\frac {3}{4}\left( p-2 \right)-1>0 $, then $\Phi \left( u \right)>0$ for any $u\in {H}^{1}(\R^3)\setminus\{0\} $.\par
Use the same argument of Lemma $2.15$ in $\cite{sx2020nonautonomous}$, we have
\begin{equation}\nonumber
	I(\bar{u})=m\left( c \right),\;P(\bar{u})=0,\; \|\bar{u} \|_{2}^{2}=c.
\end{equation}
This shows that $m\left( c \right)$ is achieved. \par
Case (ii) $\overline{u}=0$. Up to a subsequence,
\begin{equation}\nonumber
	u_n\rightharpoonup 0\text{ in }H^1(\R^3),\; u_n\to 0\text{ in }L^s_{loc}(\R^3) \text{ for }s\in [2,2^*),\; u_n\to 0\text{ a.e. }\R^3.
\end{equation}

Similarly, by Lemma \ref{1 lem3.1}, for any $\varepsilon>0$, there exists $R>0$ large enough such that
$$|A_{\infty}-A(x)|<\varepsilon~~ \forall x\in B^c_R(0).$$
Then by $u_n\to 0 \; in \; L_{loc}^q(\R^3)$ and Sobolev imbedding theorem we have
\begin{equation}\label{1 equa.2 third0limit}
		\begin{aligned}
			\left|\int_{\R^3}(A_{\infty}-A(x))|u_n|^pdx\right|\le o(1)+ C\varepsilon,
		\end{aligned}
	\end{equation}
and
\begin{equation}\label{2 equa.1 third0limit}
	\begin{aligned}
		\left|\int_{\R^3}\left[\frac{3}{2}(p-2)A(x)-\nabla A(x)\cdot x-\frac{3}{2}(p-2)A_{\infty}\right]|u_n|^pdx\right|\le o(1)+ C\varepsilon.
	\end{aligned}
\end{equation}
Therefore, by \eqref{1 equa.2 third0limit} and \eqref{2 equa.1 third0limit} and the arbitrariness of $\varepsilon$, we obtain
\begin{equation}\label{2 equ3.8. I{infty} to m(c)}
	I_{\infty}(u_n)\to m(c),\; P_{\infty}(u_n)\to 0\text{ as }n\to \infty,
\end{equation}
where $$P_{\infty}(u_n)=\dfrac{1}{2}\|\nabla u_n\|^2_2+\dfrac{1}{4}\int_{\R^3}\int_{\R^3}\frac{|u_n(x)|^2|u_n(y)|^2}{|x-y|}dxdy-\dfrac{3(p-2)}{2p}A_{\infty}\int_{\R^3}|u_n(x)|^pdx.$$
By Lemma \ref{2 m(c)>0}-(i) and \eqref{2 equ3.8. I{infty} to m(c)}, we have
\begin{equation}\label{20231126-e10}
	0<\xi_0^2 \le \|\nabla u_n\|^2_2=\dfrac{3(p-2)}{p}A_{\infty}\int_{\R^3}|u_n(x)|^pdx-\dfrac{1}{2}\int_{\R^3}\int_{\R^3}\frac{|u_n(x)|^2|u_n(y)|^2}{|x-y|}dxdy,
\end{equation}
where $\xi_0$ only depends on $c$.
\eqref{20231126-e10} implies that $\{u_n\}$ doesn't vanish. Indeed, if $$u_n\to 0\quad \text{in}~L^p(\R^3),$$ then by \eqref{20231126-e10} we see that
$$\lim\limits_{n\to\infty}P_{\infty}(u_n)\geq \frac{1}{2}\xi_0^2,$$
which contradicts with \eqref{2 equ3.8. I{infty} to m(c)}.
Hence $\{u_n\}$ doesn't vanish. Using Lions' concentration compactness principle $\cite{lionI1984,lionII1984}$, there exist $\delta>0$ and $\{y_n\} \subset \R^3$ such that
\begin{equation}\nonumber
	\int_{B_1(y_n)}|u_n|^2dx=\int_{B_1(0)}|u_n(x+y_n)|dx>\dfrac{\delta}{2}.
\end{equation}
Therefore, let $\hat{u}_n(x)=u_n(x+y_n)$, there exists $\hat{u}\in H^1(\R^3)\setminus \{0\}$ such that up to a subsequence
\begin{equation}\label{A2 equ10.4}
	\hat{u}_n\rightharpoonup \hat{u}\text{ in }H^1(\R^3),\; \hat{u}_n\to \hat{u}\text{ in }L^s_{loc}(\R^3) \text{ for }s\in [2,2^*),\; \hat{u}_n\to \hat{u}\text{ a.e. }\R^3.
\end{equation}
Let $w_n=\hat{u}_n-\hat{u}$, then \eqref{A2 equ10.4}, Lemma \ref{1 le 2.1} and Lemma \ref{1 le 2.2} yield
\begin{equation}\label{A2 equ10.5 F P'{infty}}
	I_{\infty}(\hat{u}_n)=I_{\infty}(\hat{u})+I_{\infty}(w_n)+o(1),\; P_{\infty}(\hat{u}_n)=P_{\infty}(\hat{u})+P_{\infty}(w_n)+o(1).
\end{equation}
Set
\begin{equation}\nonumber
\begin{aligned}
    \Psi_{\infty}(u):=&I_{\infty}(u)-\dfrac{1}{2}P_{\infty}(u)=\dfrac{A_{\infty}}{p}\left(\dfrac{3(p-2)}{2}-1\right)\int_{\R^3}|u|^pdx
    +\frac{1}{8}\int_{\R^3}\int_{\R^3}\frac{|\hat{u}_n(x)|^2|\hat{u}_n(y)|^2}{|x-y|}dxdy.
\end{aligned}
\end{equation}
For $u\in H^1(\R^3)\setminus \{0\}$ we see $\Psi_{\infty}(u)>0$. Thanks to \eqref{A2 equ10.4}, \eqref{A2 equ10.5 F P'{infty}} and Lemma \ref{2 m(c)>0}, we get
\begin{equation}\nonumber
\Psi_{\infty}(w_n)=m(c)-\Psi_{\infty}(\hat{u})+o(1)
\end{equation}
and
\begin{equation}\nonumber
	P_{\infty}(w_n)=-P_{\infty}(\hat{u})+o(1).
\end{equation}
Using the
same argument of Lemma $2.15$ in $\cite{sx2020nonautonomous}$, we have
\begin{equation}\nonumber
	I_{\infty}(\hat{u})=m(c),\; P_{\infty}(\hat{u})=0.
\end{equation}
By Lemma \ref{2 unique t_u}, there is $\hat{t}>0$ such that $\hat{u}^{\hat{t}}\in \mathcal{P}(c)$. Then by
(A1) and Lemma \ref{2 I(u)>I(u^t)} we obtain
\begin{equation}\nonumber
	m(c)\le I(\hat{u}^{\hat{t}}) \le I_{\infty}(\hat{u}^{\hat{t}})\le I_{\infty}(\hat{u})=m(c).
\end{equation}
This completes the conclusion.

\end{proof}

{\bf The proof of Theorem \ref{2 th}}: According to Lemma \ref{2 achieved}, for any
$c>0$, there exists $\bar{u}_c \in \mathcal{P}(c)$ such that
$$ I(\bar{u}_c)=m(c),I'|_{S(c)}(\bar{u}_c)=0.$$
Then by Lagrange multiplier theorem, there exists $\lambda_c\in \mathbb{R}$ such that
 $I'(\bar{u}_c)+\lambda_c\bar{u}_c=0.$
Therefore, we find a solution $(\bar{u}_c,\lambda_c)$ of equation $\eqref{1.1.0}$.

\vskip4mm
{\section{ Conflict of Interest }}
\setcounter{equation}{0}
The authors declared that they have no conflict of interest.

\vskip4mm

\end{document}